\documentclass[12pt]{article}
\usepackage{color, fullpage}
\usepackage{authblk,amsmath,amsthm,amsfonts,amssymb,amscd,enumerate,graphicx,color,appendix,tabularx,cite}
\allowdisplaybreaks[4]
\newtheorem {Lemma}{Lemma}[section]
\newtheorem {Theorem} {Theorem}[section]

\newtheorem {Corollary}{Corollary}[section]

\numberwithin{equation}{section}
\usepackage{fullpage}
\allowdisplaybreaks [4]
\begin{document}

\title{On ABC spectral radius of uniform hypergraphs}

\author{
Hongying Lin$^{1}$\footnote{E-mail: linhy99@scut.edu.cn}, Bo Zhou$^{2}$\footnote{Corresponding author. E-mail: zhoubo@m.scnu.edu.cn}\\
$^{1}$School of Mathematics, South China  University of Technology, \\
Guangzhou 510641, P.R. China\\
$^{2}$School of  Mathematical Sciences, South China Normal University, \\
Guangzhou 510631, P.R. China}

\date{}
\maketitle

\begin{abstract}
Given a $k$-uniform hypergraph $G$ with vertex set $[n]$ and edge set $E(G)$, the ABC tensor $\mathcal{ABC}(G)$ of $G$ is the $k$-order $n$-dimensional tensor with
\[
 \mathcal{ABC}(G)_{i_1, \dots, i_k}=
 \begin{cases}
 \dfrac{1}{(k-1)!}\sqrt[k]{\dfrac{\sum_{i\in e}d_{i}-k}{\prod_{i\in e}d_{i}}} & \mbox{if $e=\{i_1, \dots, i_k\}\in E(G)$}\\
 0 & \mbox{otherwise}
\end{cases}
\]
for $i_j\in [n]$ with $j\in [k]$, where $d_i$ is the degree of vertex $i$ in $G$.
The ABC spectral radius of a uniform hypergraph is the spectral radius of its  ABC tensor.   We give tight lower and upper bounds for the ABC spectra radius, and determine  the maximum ABC spectral radii of uniform hypertrees,  uniform non-hyperstar hypertrees and  uniform non-power  hypertrees of given size, as well as the maximum ABC spectral radii of unicyclic uniform hypergraphs and linear unicyclic uniform hypergraphs of given size, respectively. We also characterize those uniform hypergraphs for which the maxima for the ABC spectral radii are actually attained in all cases.
\\ \\
{\bf Keywords:} ABC tensor,  ABC spectral radius, uniform  hypergraph, H-eigenvalue
\end{abstract}
%
%

\section{Introduction}

Given a positive integer $k\ge 2$, a  $k$-uniform hypergraph $G$ consists of  a finite set of vertices $V(G)$ a set of hyperedges (or simply edges) and $E(G)\subseteq 2^{V(G)}$ such that each edge contains exactly $k$ vertices. We call the numbers of vertices and edges of $G$ as the order and size of $G$, respectively.
A uniform hypergraph is a $k$-uniform hypergraph for some $k$.
A linear hypergraph is one in which every two distinct edges intersect in at most one vertex.

Let $G$ be a $k$-uniform hypergraph of order $n$ with vertex set $V(G)=[n]:=\{1, \dots, n\}$. For
$i\in V(G)$, denote by $E_i(G)$ the set of edges containing $i$, and
the degree of $i$ in $G$, denoted by $d_G(i)$ or simply $d_i$, is $|E_i(G)|$. The  hypergraph
$G$ is regular  if all the degree of its vertices are equal.
Assume that $E(G)\ne \emptyset$ for any hypergraph $G$ in this paper.

For integers $k$ and $n$  with $2\le k<n$, a $k$-order $n$-dimensional complex or real tensor (or hypermatrix)  $\mathcal{T}$ is a multidimensional array of $n^k$
elements of the form
$\mathcal{T}= (\mathcal{T}_{i_1,\dots, i_k})$, where $1 \le i_1, \dots, i_k\le n$.
A $k$-order $n$-dimensional
real tensor is said to be a nonnegative tensor if all
its entries are nonnegative.
For  a $k$-order
$n$-dimensional real tensor $\mathcal{T}$ and  an $n$-dimensional vector  ${\bf x}=(x_1,\dots, x_n)^\top$, the product
$\mathcal{T}{\bf x}^{k-1}$ is defined to be an $n$-dimensional vector so that  for $i\in [n]$,
\[
(\mathcal{T}{\bf x}^{k-1})_i=\sum_{i_2\in [n]}\dots \sum_{i_k\in [n]}
\mathcal{T}_{i,i_2,\dots, i_k}x_{i_2}\dots x_{i_k},
\]
while $\mathcal{T}{\bf x}^{k}$ is defined as the following homogeneous polynomial
\[
\mathcal{T}{\bf x}^{k}=\sum_{i_1\in [n]}\dots \sum_{i_k\in [n]}\mathcal{T}_{i_1,\dots, i_k}x_{i_1}\dots x_{i_k}.
\]
So $\mathcal{T}{\bf x}^{k}={\bf x}^\top (\mathcal{T}{\bf x}^{k-1})$. Let ${\bf x}^{[k]}=(x_1^k,\dots, x_n^k)^{\top}$.

Lim \cite{Lim} and Qi \cite{Qi05} proposed independently the concepts of eigenvalues and eigenvectors of a $k$-order $n$-dimensional real tensor $\mathcal{T}$.
A complex $\lambda$ is called an eigenvalue of $\mathcal{T}$,
if the system of homogeneous polynomial equations
\[
\mathcal{T}{\bf x}^{k-1}=\lambda {\bf x}^{[k-1]},
\]
i.e.,
\[
(\mathcal{T}{\bf x}^{k-1})_i=\lambda x_i^{k-1} \mbox{ for all $i\in [n]$}
\]
has a  nonzero solution  ${\bf x}$.
The vector ${\bf x}$ is called an eigenvector of $\mathcal{T}$ corresponding to $\lambda$, and the  equalities
\[
\sum_{i_2\in [n]}\dots \sum_{i_k\in [n]}\mathcal{T}_{i,i_2,\dots, i_k}x_{i_2}\dots x_{i_k}=\lambda x_i^{k-1}
\]
for $i=1, \dots, n$
are called the $(\lambda, {\bf x})$-eigenequations of $\mathcal{T}$. Moreover, if both $\lambda$ and ${\bf x}$ are real, then we call $\lambda$ an  $H$-eigenvalue and ${\bf x}$ an $H$-eigenvector of $\mathcal{T}$, see also \cite{Q2,Q3}.
The spectral radius  of $\mathcal{T}$ is the maximum  modulus of its eigenvalues, denoted by
$\rho(\mathcal{T})$.

Let $G$ be a $k$-uniform hypergraph of order $n$.
Recall that the adjacency tensor $\mathcal{A}(G)$ of $G$ is defined as \cite{CD}
\[
 \mathcal{A}(G)_{i_1, \dots, i_k}=
 \begin{cases}
\frac{1}{(k-1)!} & \mbox{if $\{i_1, \dots, i_k\}\in E(G)$},\\
 0 & \mbox{otherwise}.
\end{cases}
\]
Fix $i\in \{1, \dots, n\}$. If $\{i, i_2, \dots, i_k\}\in E(G)$, then
$\mathcal{A}(G)_{i, \tau(i_2), \dots, \tau(i_k)}=\frac{1}{(k-1)!}$ for any permutation $\tau$ in the symmetric group of degree $k-1$, and as there are $(k-1)!$ such permutations, one has
\[
\sum_{i_2,\dots, i_k\in [n]}\mathcal{A}(G)_{i, i_2, \dots, i_k}=\sum_{e\in E_i(G)}\frac{1}{(k-1)!}(k-1)!
=\sum_{e\in E_i(G)}1 =d_i.
\]
That is, the $i$-th row sum of  $\mathcal{A}(G)$ is just the degree of the $i$-th vertex of $G$. The ABC tensor $\mathcal{ABC}(G)$ of the $k$-uniform hypergraph $G$ as the $k$-order
$n$-dimensional tensor with entries
\[
 \mathcal{ABC}(G)_{i_1, \dots, i_k}=
 \begin{cases}
 \dfrac{1}{(k-1)!}\sqrt[k]{\dfrac{\sum_{i\in e}d_{i}-k}{\prod_{i\in e}d_{i}}} & \mbox{if $e=\{i_1, \dots, i_k\}\in E(G)$},\\
 0 & \mbox{otherwise}.
\end{cases}
\]
The term `ABC' is  abbreviated from atom-bond connectivity that comes from chemistry \cite{Est}. The ABC tensor of the hypergraph $G$ may be viewed as the
adjacency tensor  of an edge-weighted hypergraph $G_w$ in which an edge $e=\{i_1,\dots, i_k\}\in E(G)$ has weight $\sqrt[k]{\dfrac{\sum_{i\in e}d_{i}-k}{\prod_{i\in e}d_{i}}}$ based on the degrees of the vertices in the edge.
For a $k$-uniform hypergraph $G$, the ABC eigenvalues of $G$ are defined as the eigenvalues of its ABC tensor, and in particular,
the ABC spectral radius of $G$ is defined as the spectral radius of its ABC tensor, denoted by $\rho_{\mathcal{ABC}}(G)$. That is, $\rho_{\mathcal{ABC}}(G)=\rho(\mathcal{ABC}(G))$.
Recall that the spectral radius of a hypergraph $G$ is the spectral radius of its adjacency tensor of $G$, denoted by $\rho_{\mathcal{A}}(G)$, see, e.g., \cite{N1,N2}.

An ordinary graph is a just $2$-uniform hypergraph, so the ABC tensor of a graph $G$ is just the ABC  matrix of $G$, which was proposed by Estrada \cite{Est} in the context of molecular graphs based on early work, see, e.g. \cite{Est1}.
For edge $\{i,j\}$,
 $\frac{d_i+d_j-2}{d_id_j}$ is interpreted
as the probability of visiting  edge $ij$ from $i$ or $j$.
Such interpretation in the context of molecular graphs is related
to the polarizing capacity of the bond considered.
Since  the work of Estrada \cite{Est},
the spectral properties of the ABC matrix of a graph have received much attention,
see, e.g. \cite{Chen2,Chen3,Gh,LW,YD}.
For a connected graph $G$ on $n\ge 3$ vertices, Ghorbani et al.~\cite{Gh} and Chen \cite{Chen3} independently showed that the path and the complete graph are the unique ones that minimize and maximizes the ABC spectral radius.
If $G$ is a tree of order $n\ge 2$, Chen \cite{Chen2} showed that $\rho_{\mathcal{ABC}}(G)\le \sqrt{n-2}$
with  equality  if and only if $G$ is the star.
If $G$ is a unicyclic graph of order $n\ge 3$,
Li and Wang \cite{LW} showed that $\rho_{\mathcal{ABC}}(G)$ is minimum (maximum, respectively)
if and only if $G$ is the cycle ($G$ is obtainable from the star by adding an edge, respectively), which was conjectured early in \cite{Gh}.
Further study of the ABC spectral radius of unicyclic graphs and bicyclic graphs may be found in \cite{YD,YZD}.

Additionally, the ABC index of the $k$-uniform hypergraph $G$ is defined as
\[
\sigma_{\mathcal{ABC}}(G)=\frac{1}{(k-1)!}\sum_{e\in E(G)}\sqrt[k]{\frac{\sum_{i\in e}d_{i}-k}{\prod_{i\in e}d_{i}}}.
\]
If $k=2$, the ABC index (abbreviated from the atom-bond connectivity
index) has been much studied, see, e.g., \cite{Da,ETRG,FGV,GF,HDW,ZX}, just to mention but a few.
Very recently, Estrada \cite{Es22}  proposed a statistical-mechanical theory, which is exemplified by deriving the ABC index (and generalizations) as well as others.

Let $G$ be an $r$-uniform hypergraph with $r\geq 2$. For integer $k> r$, the $k$-th power  of $G$, denoted by $G^k$, is defined to be
the $k$-uniform hypergraph with edge set $E(G^k)=\{e\cup \{v_{e,1},\ldots,v_{e,k-r}\}: e\in E(G)\}$ and
vertex set $V(G^k)=V(G)\cup \{v_{e,i}: e\in E(G),i=1,\ldots,k-r\}$, where
$v_{e,i}\ne v_{f,j}$ for any $\{e,f\}\subseteq  E(G)$ and $i,j\in \{1, \dots, k-r\}$.
 Let $G^r=G$.
A power hypergraph is a $k$-uniform hypergraph for some $k\ge 3$ such that it is the $k$-th power of some ordinary graph \cite{HQS}. A hypergraph is a non-power hypergraph if it is not a power hypergraph.

Hypergraph theory found applications in chemistry \cite{GK,KS,KS}.
The molecular structures with polycentric delocalized bonds my be represented by
hypergraphs \cite{KS}, where vertices  correspond to individual atoms,
edges of cardinality at least three correspond to delocalized polycentric bonds, and edges  of cardinality
two  correspond to simple covalent bonds.
This avoided defects peculiar for ordinary molecular graphs and
facilitated the task of comparing the ordinary
molecular structures with the structures containing polycentric
bonds.  By comparative analysis of topological and information indices for eight series of molecular structures in \cite{KS2},  it was demonstrated that the hypergraph model gives a higher accuracy
of molecular structure description.

In this article, we extend the study of the ABC spectral properties of ordinary graphs from begun with Estrada \cite{Est}
to the study of the more general ABC spectral properties  of uniform hypergraphs.
For the ABC spectral properties of  $k$-uniform hypergraphs, there are difference for the case $k\ge 3$ and the case $k=2$ (see Section 7 below).
Generalizing  and extending the ABC spectral properties from
ordinary graphs to uniform hypergraphs,
we establish tight lower and upper bounds for the ABC spectral radius of $k$-uniform hypergraphs, and
determine  the $k$-uniform hyppertree of fixed size with first and second  maximum ABC spectral radii
and the non-power $k$-uniform hyppertree of fixed size with maximum ABC spectral radius, as well as
the unique $k$-uniform unicyclic hyergraph of fixed size with maximum ABC spectral radius and the linear $k$-uniform unicyclic hypergraph of fixed size with maximum ABC spectral radius. We list the main results as below.

\begin{Theorem} \label{LUpper}
Let $G$ be a connected  $k$-uniform hypergraph of order $n$. Then
\[
\min\left\{\sqrt[k]{\sum_{i\in e}d_{i}-k}: e\in E(G)\right\}\le \rho_{\mathcal{ABC}}(G)\le \max\left\{\sqrt[k]{\sum_{i\in e}d_{i}-k}: e\in E(G)\right\}
\]
with either equality if and only if the sum of degrees of vertices from each edge is a constant.
\end{Theorem}

\begin{Corollary}
Let $G$ be a connected  $k$-uniform hypergraph of order $n$ with minimum degree $\delta$ and maximum degree $\Delta$. Then
\[
\sqrt[k]{k\delta-k}\le \rho_{\mathcal{ABC}}(G)\le \sqrt[k]{k\Delta-k}
\]
with either equality if and only if $G$ is regular.
\end{Corollary}

\begin{Theorem} \label{med}
Let $G$ be a connected  $k$-uniform hypergraph with maximum degree $\Delta\ge 2$.
 Then
\[
\rho_{\mathcal{ABC}}(G)\le \sqrt[k]{\frac{\Delta-1}{\Delta}} \rho_{\mathcal{A}}(G)
\]
with equality if and only if
$\omega_G(e)=\frac{\Delta-1}{\Delta}$ for any $e\in E(G)$.
\end{Theorem}

By previous theorem, any upper bound on the spectral radius of the adjacency tensor will lead to an upper bound on the ABC spectral radius.

A hypertree is connected hypergraph without cycles. An edge in a hypergraph is a pendant edge if it contains at most one vertex of degree greater than one.
A hyperstar is a hypertree
of which every edge is a pendant edge.
Denote by $S_{m,k}$ the $k$-uniform hyperstar with $m$ edges.
The center of $S_{m,k}$ is defined as the vertex of degree $m$ in $S_{m,k}$.
For $m\geq 3$ and $1\leq a\leq \frac{m-1}{2}$, let $D_{m,a}$ be the double star
obtained by adding an edge between the centers of two disjoint stars $S_{a,2}$ and $S_{m-1-a,2}$.

\begin{Theorem}\label{abc7-180-2}
For $k\geq 2$, let $G$ be a $k$-uniform hypertree of size $m\ge 1$. Then
\[
\rho_{\mathcal{ABC}}(G)\le \sqrt[k]{m-1}
\]
with equality if and only if $G\cong S_{m,k}$. Moreover, if $G$ is different from $S_{m,k}$, then
\[
\rho_{\mathcal{ABC}}(G)\leq \sqrt[k]{\frac{m^2-3m+3+\sqrt{(m-1)^2+(m-2)^4}}{2(m-1)}}
\]
with equality if and only if $G\cong D^{k}_{m,1}$.
\end{Theorem}

Let $S_{m,k;m-3,1,1}$ be the $k$-uniform hypertree obtained from $S_{m-2, k}$ with two chosen vertices of degree one in a common edge by adding a pendant edge
at each of them.

\begin{Theorem}\label{abc7-180-01}
For $k\geq 3$,
let $G$ be a non-power $k$-uniform hypertree of size $m\ge 4$. Then
\[
\rho_{\mathcal{ABC}}(G)\leq b_m^\frac{1}{k}
\]
with equality if and only if $G\cong S_{m,k;m-3,1,1}$,
where  $b_m$ is the largest root of
with
\[
4(m-2)t^{3}-(4m^2-19m+27)t^{2}+(4m^2-23m+34)t-(m-3)^2=0.
\]
\end{Theorem}

For $g\ge 2$ and $k\ge 3$,
 a $k$-uniform hypercycle of length $g$, denoted by $C_{g,k}$, is a $k$-uniform hypergraph  whose vertices may be labelled as $v_1,\dots, v_{g(k-1)}$ so that the set is $\{e_1, \dots, e_{g}\}$, where
$e_i=\{v_{(i-1)(k-1)+1},\dots, v_{i(k-1)+1}\}$
for $i=1, \dots, g$ with $v_{g(k-1)+1}\equiv v_1$.
For $m\geq 2$,  $k\geq 3$ and $g=2,3$,
let $U_{m,g}^{(k)}$ be the $k$-uniform unicyclic hypergraph obtained from a $k$-uniform hypercycle of length  $g$ by adding
$m-g$ pendant edges at a vertex of degree $2$.

\begin{Theorem}\label{abc7-18-1}
For $k\geq 3$, let $G$ be a $k$-uniform unicyclic hypergraph of size $m\ge 2$. Then
\[
\rho_{\mathcal{ABC}}(G)\leq \sqrt[k]{m-1+\frac{2}{m}}
\]
with equality if and only if $G\cong U^{(k)}_{m,2}$.
\end{Theorem}

\begin{Theorem}\label{abc9-7-1}
For $k\geq 3$, let $G$ be a linear $k$-uniform unicyclic hypergraph of size $m\ge 3$. Then
\[
\rho_{\mathcal{ABC}}(G)\le a_m^{\frac{2}{k}}
\]
with equality if and only if $G\cong U_{m,3}^{(k)}$, where $a_m$ is the largest root of
\[
t^3-\frac{\sqrt{2}}{2}t^2-\frac{m^2-4m+ 5}{m - 1}t+\frac{\sqrt{2}(m^2- 5m+ 6)}{2(m - 1)}=0.
\]
\end{Theorem}

Unlike the linear nature of spectral properties of various matrices associated to a graph, the spectral properties of tensors associated to a hypergraph have a nonlinear dependence on the adjacency of the hypergraph.

\section{Preliminaries}

Let $G$ be a hypergraph with $v,w\in V(G)$.
A path from $v$ to $w$ in $G$ is a set of  distinct
vertices $v_1,\dots, v_{\ell+1}$ and a set of  distinct edges $e_1, \dots, e_{\ell}$  for some $\ell$ such that for $i=1, \dots, \ell$,  $\{v_i, v_{i+1}\}\subseteq  e_i$,  and for $j>i+1$, $e_i\cap e_j=\emptyset$, where $v_1=v$ and $v_{\ell}=w$.
A hypergraph $G$ is connected if for every pair of vertices $u, v\in V(G)$, there is a path from $u$ to $v$ in $G$.
A cycle in $G$ is a set of distinct
vertices $v_1,\dots, v_{\ell}$ and a set of distinct edges $e_1, \dots, e_{\ell}$  for some $\ell\ge 2$ such that for $i=1, \dots, \ell$,  $\{v_{i}, v_{i+1}\}\subseteq  e_i$ (with $v_{\ell+1}\equiv v_1$), and for $|i-j|>1$, $e_i\cap e_j=\emptyset$ (with $e_{\ell+1}\equiv e_1$). A hypertree is connected hypergraph without cycles. A unicyclic hypergraph is connected hypergraph with exactly one cycle.
An ordinary tree is a $2$-uniform hypertree. Generally, an ordinary unicyclic graph is a $2$-uniform unicyclic hypergraph in which  the length of its unique cycle is at least three.

A pendant vertex of a hypergraph is a vertex of degree one. Let $G$ be a $k$-uniform hypergraph with $u\in V(G)$ and $u_i\not\in V(G)$ for $i=2,\dots,k$. The hypergraph with vertex set $V(G)\cup\{u,u_2,\dots, u_k\}$ and edge set $E(G)\cup \{u, u_2, \dots, u_k\}$ is said to be obtained from $G$ by adding a new pendant edge $\{u, u_2, \dots, u_k\}$ at $u$.

A nonnegative $k$-order $n$-dimensional tensor $\mathcal{T}$ is said to be weakly irreducible \cite{FGH,PT} if for any $J$ with  $\emptyset\ne J\subset [n]$, there is at least one entry $\mathcal{T}_{i_1, \dots, i_k}\ne 0$ with $i_1\in J$ and $i_j\in \{1, \dots, n\}\setminus J$ for some $j=2, \dots, k$.

We need the following lemmas.
The first lemma is the  Perron-Frobenius Theorem for nonnegative tensors, see \cite[Theorem 1.3]{CPZ}, \cite[Theorem 2.3]{YY}, and \cite[Theorem 4.1]{FGH}.

\begin{Lemma} \label{PF} Let $\mathcal{T}$ be a $k$-order $n$-dimensional nonnegative tensor. Then
\begin{enumerate}
\item[(i)]
$\rho(\mathcal{T})$ is an $H$-eigenvalue with a nonnegative eigenvector.

\item[(ii)]
If $\mathcal{T}$ is weakly irreducible, then $\rho(\mathcal{T})$ is an $H$-eigenvalue with a positive eigenvector and no other
eigenvalue has a positive eigenvector.
\end{enumerate}
\end{Lemma}

Let $\mathcal{T}$ be  a nonnegative $k$-order $n$-dimensional tensor.
Lemma \ref{PF}(i) says that $\rho(\mathcal{T})$ is an $H$-eigenvalue of $\mathcal{T}$, and there is a nonnegative eigenvector corresponding to $\rho (\mathcal{T})$. A nonnegative $n$-dimensional vector ${\bf x}$ is $k$-unit  $\sum_{i=1}^n x_i^k=1$. Lemma \ref{PF}(ii) says that  if  $\mathcal{T}$ is weakly irreducible, then there is a unique $k$-unit positive vector associated with $\rho (\mathcal{T})$.

Let $\mathcal{T}_1$ and $\mathcal{T}_2$ be $k$-order $n$-dimensional real tensors. If  $\mathcal{T}_2-\mathcal{T}_1$ is nonnegative, then we write $\mathcal{T}_1\le \mathcal{T}_2$.
The following lemma is \cite[Theorem 3.4]{RSB} (see also \cite[Lemma 2.3]{HQ}).

\begin{Lemma} \label{Tcom} \cite{RSB}
Let $\mathcal{T}_1$ and $\mathcal{T}_2$ be nonnegative $k$-order $n$-dimensional tensors such that
$\mathcal{T}_1\le \mathcal{T}_2$ and $\mathcal{T}_2$ is weakly irreducible. Then  $\rho(\mathcal{T}_1)\le \rho(\mathcal{T}_2)$. Moreover, if $\mathcal{T}_1\ne \mathcal{T}_2$, then  $\rho(\mathcal{T}_1)<\rho(\mathcal{T}_2)$.
\end{Lemma}

Let $G$ be a $k$-uniform hypergraph of order $n$.
It is known that $\mathcal{A}(G)$ is weakly irreducible if and only if $G$ is connected~\cite{PT}, so
$\mathcal{ABC}(G)$ is weakly irreducible if and only if $G$ is connected. Thus, if $G$ is connected,
then $\rho_{\mathcal{ABC}}(G)$ is the maximum $H$-eigenvalue of $\mathcal{ABC}(G)$, and  there is  a unique $k$-unit positive vector corresponding to  $\rho_{\mathcal{ABC}}(G)$.

\begin{Lemma} \label{R-qe} \cite{Qi05}
Let $\mathcal{T}$ be a $k$-order $n$-dimensional symmetric nonnegative tensor. Then, for any nonnegative $n$-dimensional $k$-unit column vector ${\bf x}$,
 \[
 \rho(\mathcal{T})\ge \mathcal{T}{\bf x}^k
 \]
with equality when $\mathcal{T}$ is weakly irreducible if and only if ${\bf x}$ is the unique $k$-unit positive eigenvector
corresponding to $\rho(\mathcal{T})$.
\end{Lemma}

Let $G$ be a $k$-uniform hypergraph of order $n$. The Randi\'{c} tensor $\mathcal{R}(G)$ of $G$ is defined to be the tensor of order $k$ and dimension $n$ with entries
\[
 \mathcal{R}(G)_{i_1, \dots, i_k}=
 \begin{cases}
\frac{1}{(k-1)!\sqrt[k]{\prod_{j=1}^kd_{i_j}}} & \mbox{if $\{i_1, \dots, i_k\}\in E(G)$},\\
 0 & \mbox{otherwise}.
\end{cases}
\]
It is the normalized adjacency tensor in \cite{HQ}. If $k=2$, it is just the Randi\'c matrix \cite{GFB}.
The following lemma is an extension of \cite[Theorem 2.3]{GFB}, see also a different treatment in  \cite[pp.~2--4]{Chung}.

\begin{Lemma} \label{randic} Let $G$ be a nontrivial connected  $k$-uniform hypergraph of order $n$. Then
\[
\rho(\mathcal{R}(G))=1.
\]
\end{Lemma}

\begin{proof} Let ${\bf x}=(\sqrt[k]{d_1}, \dots, \sqrt[k]{d_n})^\top$. For $i_1=1, \dots, n$,
if $\{i_1, i_2, \dots, i_k\}\in E(G)$, then
\[
\mathcal{R}(G)_{i_1, \tau(i_2), \dots, \tau(i_k)}=\frac{1}{(k-1)!}\frac{1}{\sqrt[k]{\prod_{j=1}^kd_{i_j}}}
 \]
for any permutation $\tau$ in the symmetric group of degree $k-1$, and as there are $(k-1)!$ such permutations, one has
\[
\begin{split}
(\mathcal{R}(G){\bf x}^{k-1})_{i_1}
=& \sum_{i_2\in [n]}\dots \sum_{i_k\in [n]}
\mathcal{R}(G)_{i_1,i_2, \dots, i_k}\prod_{j=2}^kx_{i_j}\\
=& \sum_{e\in E_{i_1}(G)}\frac{1}{\sqrt[k]{\prod_{j=1}^kd_{i_j}}}\prod_{j=2}^k\sqrt[k]{d_{i_j}} \\
=& \sum_{e\in E_{i_1}(G)} \frac{1}{\sqrt[k]{d_{i_1}}} \\
=& \frac{1}{\sqrt[k]{d_{i_1}}}\sum_{e\in  E_{i_1}(G)}1 \\
=&(\sqrt[k]{d_{i_1}})^{k-1},
\end{split}
\]
so $\mathcal{R}(G){\bf x}^{k-1}={\bf x}^{[k-1]}$. It follows that $1$ is an eigenvalue of $\mathcal{R}(G)$ with a positive eigenvector ${\bf x}$.
As $G$ is connected,  $\mathcal{R}(G)$ is weakly irreducible, so by Lemma~\ref{PF}, we have $\rho(\mathcal{R}(G))=1$.
\end{proof}

\begin{Lemma}\label{ABC-permutation} 
Let $G$ be a connected $k$-uniform hypergraph. Let ${\bf x}$ be the $k$-unit positive eigenvector
of $\mathcal{A}(G)$ ($\mathcal{ABC}(G)$, respectively)) corresponding to  $\rho_{\mathcal{A}}(G)$ ($\rho_{\mathcal{ABC}}(G)$, respectively).
Let $\sigma$ be an automorphism of $G$.
Then  $x_u=x_v$ provided that  $\sigma(u)=v$.
\end{Lemma}

\begin{proof} Suppose that $P$ is the  permutation matrix  corresponding to the automorphism $\sigma$ of $G$, i.e., $P_{ij}=1$ if and only if $\sigma(i)=j$ for $i\in V(G)$. Let $\mathcal{Q}\in \{\mathcal{A}(G), \mathcal{ABC}(G)\}$. Then  $\mathcal{Q}(G)=P \mathcal{Q}(G) P^\top$.
So
\[
{\bf x}^\top (\mathcal{Q}(G) {\bf x})= {\bf x}^\top (P \mathcal{Q}(G) P^\top {\bf x})= (P^\top  {\bf x})^\top \mathcal{Q}(G) P^\top  {\bf x}.\]
Note that $P^\top  {\bf x}$ is positive and $\sum_{i\in V(G)} y_i^k=\sum_{i\in V(G)} x_i^k=1$, where   $\mathbf{y}=P^\top  {\bf x}$.
So  $P^\top  {\bf x}$ is also a $k$-unit positive eigenvector corresponding to $\rho(\mathcal{Q})$. By Lemma~\ref{PF} (ii),  one has
 $P^\top  {\bf x}= {\bf x}$.
Hence, $x_u=x_v$  if   $\sigma(u)=v$.
\end{proof}

\section{Bounds for ABC spectral radius}

For an edge $e$ of a $k$-uniform hypergraph $G$, set $\omega_G(e)=\frac{\sum_{i\in e}d_i-k}{\prod_{i\in e}d_i}$.

\begin{Theorem} \label{low1}
Let $G$ be a connected  $k$-uniform hypergraph of order $n$. Then
\[
\rho_{\mathcal{ABC}}(G)\ge n^{-1}k!\sigma_{\mathcal{ABC}}(G)
\]
with equality if and only if $\sum_{e\in E_i(G)}\sqrt[k]{\omega_G(e)}$ is a constant for $i=1,\dots, n$.
\end{Theorem}

\begin{proof} 
Setting ${\bf x}$ to be the $k$-unit $n$-dimensional vector $n^{-\frac{1}{k}}(1, \dots, 1)^\top$, we have
\[
\begin{split}
(\mathcal{ABC}(G){\bf x}^{k-1})_i=& \sum_{i_2\in [n]}\dots \sum_{i_k\in [n]}
\mathcal{ABC}(G)_{i,i_2, \dots, i_k}\prod_{j=2}^kx_{i_j}\\
=& \sum_{i_2\in [n]}\dots \sum_{i_k\in [n]}\mathcal{ABC}(G)_{i,i_2, \dots, i_k}n^{-\frac{k-1}{k}}\\
=&  n^{-\frac{k-1}{k}}\sum_{e\in E_i(G)}\frac{1}{(k-1)!}\sqrt[k]{\omega_G(e)}(k-1)!\\
=& n^{-\frac{k-1}{k}}\sum_{e\in E_i(G)}\sqrt[k]{\omega_G(e)},
\end{split}
\]
so
\[
\begin{split}
\mathcal{ABC}(G){\bf x}^k  =& {\bf x}^{\top}
(\mathcal{ABC}(G){\bf x}^{k-1})\\
=& \sum_{i=1}^n n^{-\frac{1}{k}}n^{-\frac{k-1}{k}}\sum_{e\in E_i(G)}\sqrt[k]{\omega_G(e)}\\
=& n^{-1} \sum_{i=1}^n\sum_{e\in E_i(G)}\sqrt[k]{\omega_G(e)}\\
=& n^{-1} k\sum_{e\in E(G)}\sqrt[k]{\omega_G(e)}\\
=&n^{-1}k!\sigma_{\mathcal{ABC}}(G).
\end{split}
\]
Note that $\mathcal{ABC}(G)$ is symmetric. So, by Lemma~\ref{R-qe}, we have
$\rho_{\mathcal{ABC}}(G)\ge n^{-1}k!\sigma_{\mathcal{ABC}}(G)$
with equality if and only if
\[
n^{-\frac{k-1}{k}}\sum_{e\in E_i(G)}\sqrt[k]{\omega_G(e)}=\rho_{\mathcal{ABC}}(G)\left(n^{-\frac{1}{k}}\right)^{k-1},
\]
i.e., $\sum_{e\in E_i(G)}\sqrt[k]{\omega_G(e)}$ is a constant for $i=1,\dots, n$.
\end{proof}

\begin{proof}[{\bf Proof of Theorem \ref{LUpper}}]
Let
\[
c=\min\left\{\sqrt[k]{\sum_{i\in e}d_{i}-k}: e\in E(G)\right\}
\]
and
\[
C=\max\left\{\sqrt[k]{\sum_{i\in e}d_{i}-k}: e\in E(G)\right\}.
\]
Then
\[
c\mathcal{R}(G)\le \mathcal{ABC}(G)\le C\mathcal{R}(G).
\]
As $G$ is connected, $\mathcal{ABC}(G)$ and $C\mathcal{R}(G)$ are weakly irreducible. So, by Lemmas \ref{Tcom} and \ref{randic}, we have
\[
 c=c\rho(\mathcal{R}(G))=\rho(c\mathcal{R}(G))\le \rho_{\mathcal{ABC}}(G)\le \rho(C\mathcal{R}(G))=C\rho(\mathcal{R}(G))=C.
\]

Suppose that $\rho_{\mathcal{ABC}}(G)=a$, where $a=c$, or $C$. By Lemma \ref{Tcom}, $\mathcal{ABC}(G)=a\mathcal{R}(G)$, so $\sqrt[k]{\sum_{i\in e}d_{i}-k}=a$ for any edge $e$ of $G$, i.e., $\sum_{i\in e}d_{i}=a^k+k$ is a  constant for any edge $e$ of $G$. Conversely, if $\sum_{i\in e}d_{i}$ is a  constant for any edge $e$ of $G$, then $c=C$, so $c\mathcal{R}(G)=\mathcal{ABC}(G)= C\mathcal{R}(G)$, implying that $c=\rho_{\mathcal{ABC}}(G)=C$.
\end{proof}

This result extends \cite[Theorem 2.8]{Chen3} from graphs to hypergraphs.

Denote by $K_n^{(k)}$ the complete $k$-uniform hypergraph, that is the hypergraph with vertex set $\{1, \dots, n\}$ such that any $k$ vertices form an edge.

\begin{Corollary} \label{easy}
Let $G$ be a connected  $k$-uniform hypergraph of order $n$. Then
\[
\rho_{\mathcal{ABC}}(G)\le \sqrt[k]{k{n-1\choose k-1}-k}
\]
with equality if and only if $G\cong K_n^{(k)}$.
\end{Corollary}

\begin{Lemma} \label{Eas} Let $a_1, \dots, a_k$ be positive integers, where $a_1\ge 2$ and $k\ge 2$. Let
\[
f(a_1, \dots, a_k)=\frac{\sum_{i=1}^ka_i-k}{\prod_{i=1}^ka_i}.
\]
Then
\[
f(a_1, \dots, a_k)< f(a_1, \dots, a_{k-1}).
\]
\end{Lemma}

\begin{proof} It is evident that $\sum_{i=1}^{k-1}a_i\ge k$, so
\[
a_k\sum_{i=1}^{k-1}a_i-a_kk\ge \sum_{i=1}^{k-1}a_i-k,
\]
i.e.,
\[
\sum_{i=1}^{k}a_i-k\le a_k\sum_{i=1}^{k-1}a_i-a_k(k-1).
\]
So the result follows.
\end{proof}

\begin{proof}[{\bf Proof of Theorem \ref{med}}] For any edge $e$ of $G$,  we have
\[
\omega_G(e)\le \frac{\max\{d_i: i\in e\}-1}{\max\{d_i: i\in e\}}\le \frac{\Delta-1}{\Delta}
\]
by Lemma~\ref{Eas} and the fact that $\frac{t-1}{t}$ is strictly increasing when $t\ge 2$. So $\mathcal{ABC}(G)\le \sqrt[k]{\frac{\Delta-1}{\Delta}}\mathcal{A}(G)$. Note that
$\mathcal{A}(G)$ is weakly irreducible.
So, by Lemma~\ref{Tcom}, we have
$\rho_{\mathcal{ABC}}(G)\le \sqrt[k]{\frac{\Delta-1}{\Delta}} \rho_{\mathcal{A}}(G)$ with equality if and only if
$\mathcal{ABC}(G)=\sqrt[k]{\frac{\Delta-1}{\Delta}}\mathcal{A}(G)$, i.e., $\omega_G(e)=\frac{\Delta-1}{\Delta}$ for any $e\in E(G)$.
\end{proof}

By previous theorem, any upper bound on the spectral radius of the adjacency tensor will lead to an upper bound on the ABC spectral radius.

\section{ABC eigenvalues of of power hypergraphs}

In ~\cite{ZSWB}, it is shown that the adjacency eigenvalues of a power hypergraph of a graph is determined by
the adjacency eigenvalues of the graph.
In  the following theorem,  we establish a relation between the ABC eigenvalues of an $r$-uniform hypergraph and the ABC eigenvalues of its $k$-th power hypergraph, where $2\leq r<k$.

\begin{Theorem}\label{power}
For $k>r\geq 2$, let $G$ be an $r$-uniform hypergraph, and $\rho$ be a nonzero  ABC eigenvalue of $G$.
Then
$\rho^\frac{r}{k}$ is an ABC eigenvalue of $G^k$.
\end{Theorem}

\begin{proof}
Let ${\bf x}$ be a nonzero eigenvector corresponding to the ABC eigenvalue $\rho$ of $G$.
By the $(\rho, {\bf x})$-eigenequations of $\mathcal{ABC}(G)$, we have
\begin{equation} \label{AB}
\rho x_i^{r-1}=\sum_{e\in E_i(G)} \sqrt[r]{\dfrac{\sum_{j\in e}d_j-r}{\prod_{j\in e} d_j}} \prod_{j\in e\setminus\{i\}}x_j
\end{equation}
for each $i\in V(G)$.
Let $\mathbf{y}$ be a column vector of dimension $|V(G^k)|$
such that
\[
y_i=\begin{cases}
 x_i^{\frac{r}{k}}& \text{if $i\in V(G)$,} \\
 \left(\dfrac{\sum_{j\in e}d_j-r}{\prod_{j\in e} d_j}\right)^{\frac{1}{rk}}\left(\dfrac{\prod_{j\in e}x_j}{\rho}\right)^{\frac{1}{k}} & \text{if $i \in \{v_{e,s} : s=1,\ldots,k-r\}$}\\
& \text{for some $e\in E(G)$.}
 \end{cases}
\]
Recall that any $e\in E(G)$ corresponds naturally to $\widetilde{e}=e\cup\{v_{e,1}, \dots, v_{e, k-r}\}$.
We show that $(\mathcal{ABC}(G^k)\mathbf{y}^{k-1})_i=\rho^{\frac{r}{k}}y_i^{k-1}$ for all $i\in V(G^k)$.
If $i\in V(G)$,
then, bearing in mind \eqref{AB}, we have
\begin{align*}
&\quad (\mathcal{ABC}(G^k)\mathbf{y}^{k-1})_i\\
&=\sum_{\widetilde{e}\in E_i(G^k)}
\sqrt[k]{\dfrac{\sum_{j\in \widetilde{e}}d_j-k}{\prod_{j\in \widetilde{e}} d_j}}
\prod_{j\in \widetilde{e}\setminus\{i\}}y_j\\
&= \sum_{e\in E_i(G)}
\sqrt[k]{\dfrac{\sum_{j\in e}d_j-r}{\prod_{j\in e} d_j}}
\prod_{j\in e\setminus\{i\}}y_j\prod_{j\in \widetilde{e}\setminus e}y_j\\
&=\sum_{e\in E_i(G)}\left(\dfrac{\sum_{j\in e}d_j-r}{\prod_{j\in e} d_j}\right)^{\frac{1}{k}}
\left(\prod _{j\in e\setminus\{i\}}x_j^{\frac{r}{k}}\right)  \left(\left(\dfrac{\sum_{j\in e}d_j-r}{\prod_{j\in e} d_j}\right)^{\frac{1}{rk}}\left(\dfrac{\prod_{j\in e}x_j}{\rho}\right)^{\frac{1}{k}}\right)^{k-r}\\
&=\dfrac{x_i^{\frac{k-r}{k}}}{\rho^{\frac{k-r}{k}}}\sum_{e\in E_i(G)} \left(\dfrac{\sum_{j\in e}d_j-r}{\prod_{j\in e} d_j}\right)^{\frac{1}{r}} \left(\prod_{j\in e\setminus\{i\}}x_j\right)  \\
&= \dfrac{x_i^{\frac{k-r}{k}}}{\rho^{\frac{k-r}{k}}}\cdot \rho x_i^{r-1}\\
&=\rho^{\frac{r}{k}} x_i^{\frac{r(k-1)}{k}}\\
&=\rho^{\frac{r}{k}} y_i^{k-1}.
\end{align*}
If $i\in V(G^k)\setminus V(G)$, then there is some $e\in E(G)$ such that $i=v_{e,s}$ for some $s=1, \dots, k-r$,
so
\begin{align*}
&\quad (\mathcal{ABC}(G^k)\mathbf{y}^{k-1})_i\\
&=
\left(\dfrac{\sum_{j\in e}d_j-r}{\prod_{j\in e} d_j}\right)^{\frac{1}{k}}
\prod_{j\in e}y_j\prod_{j\in \widetilde{e}\setminus e\setminus\{i\}}y_j\\
&=\left(\frac{\sum_{j\in e}d_j-r}{\prod_{j\in e} d_j}\right)^{\frac{1}{k}} \left(\prod_{j\in e} x_j^{\frac{r}{k}} \right) \left(\left(\dfrac{\sum_{j\in e}d_j-r}{\prod_{j\in e} d_j}\right)^{\frac{1}{rk}}\left(\dfrac{\prod_{j\in e}x_j}{\rho}\right)^{\frac{1}{k}}\right)^{k-r-1}\\
&=\left(\dfrac{\sum_{j\in e}d_j-r}{\prod_{j\in e} d_j}\right)^{\frac{k-1}{rk}} \dfrac{\prod_{j\in e} x_j^{\frac{k-1}{k}}}{\rho^{\frac{k-r-1}{k}}} \\
&=\rho^{\frac{r}{k}} \left(\left(\dfrac{\sum_{j\in e}d_j-r}{\prod_{j\in e} d_j}\right)^{\frac{1}{rk}}\left(\dfrac{\prod_{j\in e} x_j}{\rho}\right)^{\frac{1}{k}} \right)^{k-1}\\
&=\rho^{\frac{r}{k}} y_i^{k-1}.
\end{align*}
Thus $\mathcal{ABC}(G^k)\mathbf{y}^{k-1}=\rho^{\frac{r}{k}}\mathbf{y}^{[k-1]}$. From the construction of $\mathbf{y}$, $\mathbf{y}$ is a nonzero vector.
It so follows that $\rho^\frac{r}{k}$ is an ABC eigenvalue of $G^k$.
\end{proof}

For $m\geq g\geq 3$,
let $U_{m,g}$ be the unicyclic graph obtained from a cycle of length $g$ by adding $m-g$ pendant edges at a vertex of the cycle.


\begin{Corollary}\label{abc7-18-0}
For $k\geq 3$, let $G$ be a $k$-uniform power unicyclic hypergraph of size $m\ge 3$. Then
\[
\rho_{\mathcal{ABC}}(G)\le a_m^{\frac{2}{k}}
\]
with equality if and only if $G\cong U_{m,3}^k$, where $a_m$ is the largest root of
$f(t)=0$, and
\begin{eqnarray}\label{eq9-7}
f(t)=t^3-\frac{\sqrt{2}}{2}t^2-\frac{m^2-4m+ 5}{m - 1}t+\frac{\sqrt{2}(m^2- 5m+ 6)}{2(m - 1)}.
\end{eqnarray}
\end{Corollary}
\begin{proof} Let $v_1v_2v_3$ be the cycle such that $v_1$ is of degree $m-1$ in $U_{m,3}$,
and $v_0$ be a pendant vertex at $v_1$.
Let ${\bf x}$ be the $2$-unit positive eigenvector of $\mathcal{ABC}(U_{m,3})$ corresponding to $\rho=\rho_{\mathcal{ABC}}(U_{m,3})$.
Let $x_i=x_{v_i}$ for $i=0,1,2$.
By Lemma~\ref{ABC-permutation} and  the $(\rho, {\bf x})$-eigenequations of $U_{m,3}$, we have
\begin{align*}
\rho x_0&=\sqrt{\frac{m-2}{m-1}}x_1,\\
\rho x_1&=(m-3)\sqrt{\frac{m-2}{m-1}}x_0+2\sqrt{\frac{1}{2}} x_2,\\
\rho x_2&=\sqrt{\frac{1}{2}}x_1+ \sqrt{\frac{1}{2}}x_2.
\end{align*}
Since ${\bf x}$ is nonzero,
the above homogeneous linear system in the variables $x_0, x_1, x_2$
has a  nontrivial solution.
Then the determinant of its coefficient matrix is zero.
By direct calculation,  the determinant is equal to $f(\rho)=0$.
So $\rho(U_{m,3})$ is the largest root of $f(t)=0$. Now the result follows from
 Theorem~\ref{power} and the known result that $\rho_{\mathcal{ABC}}(G)\leq \rho_{\mathcal{ABC}}(U_{m,3})$
with equality if and only if $G\cong U_{m,3}$  for a unicyclic graph $G$ of size $m\geq 3$
\cite{LW,YZD}.
\end{proof}

\section{ABC spectral radius of hypertrees}

For $m\geq 4$, $m\geq k\geq 3$, $m-3\geq a_1\geq \dots \geq a_k\geq 0$ and $\sum_{i=1}^ka_i=m-1$, let $S_{m,k;a_1,\ldots,a_k}$
be the $k$-uniform hypergraph obtained by adding $a_i$ pendant edges
at $v_i$ in an edge $\{v_1,\ldots,v_k\}$. If $s$ with $1\leq s\leq k$ is the largest number such that $a_s>0$,
we write $S_{m,k;a_1,\ldots,a_s}$ instead of $S_{m,k;a_1,\ldots,a_k}$.

\begin{Lemma}\label{tree}\cite{Yuan}
For $k\geq 2$, let $G$ be a $k$-uniform hypertree of size $m\geq 5$ different from $S_{m,k}$ and $D_{m,1}^k$. Then

(i)
\[
\rho_{\mathcal{A}}(G)\leq  \rho_{\mathcal{A}}(D_{m,2}^k)<\rho_{\mathcal{A}}(D_{m,1}^k)<\rho_{\mathcal{A}}(S_{m,k})
\]
with equality if and only if $G\cong D_{m,2}^k$.

(ii) For $k\geq 3$ and $m\geq 6$, if $G$ is non-power and $G\ncong S_{m,k;m-3,1,1}$,
then
\[
\rho_{\mathcal{A}}(G)\leq  \rho_{\mathcal{A}}(S_{m,k;m-4,2,1})<\rho_{\mathcal{A}}(S_{m,k;m-3,1,1})
\]
with equality if and only if $G\cong S_{m,k;m-4,2,1}$.
\end{Lemma}

\begin{proof}[{\bf Proof of Theorem~\ref{abc7-180-2}}]
The first part follows from or from
Theorem \ref{LUpper}.

Next, we calculate $\rho_{\mathcal{A}}((D^k_{m,2}))$ and $\rho_{\mathcal{ABC}}(D^k_{m,1})$.

Let $v_1v_2v_3v_4$ be a path of $D_{m,2}$ such that $v_2$ and $v_3$ are of degrees $m-2$ and $3$, respectively.
Let ${\bf x}$ be the $2$-unit positive eigenvector of $\mathcal{A}(D_{m,2})$ corresponding to $\rho=\rho_{\mathcal{A}}(D_{m,2})$.
Let $x_i=x_{v_i}$ for $i=1,\ldots,4$.
By Lemma~\ref{ABC-permutation} and the $(\rho, {\bf x})$-eigenequations of $\mathcal{A}(D_{m,2})$, we have
\begin{align*}
\rho x_1&=x_2,\\
\rho x_2&=(m-3)x_1+x_3,\\
\rho x_3&=x_2+2x_4,\\
\rho x_4&=x_3.
\end{align*}
Then $\rho$ is the largest root of the determinant of the coefficient matrix of
the above homogeneous linear system in the variables $x_1, x_2, x_3, x_4$.
i.e., $\rho_{\mathcal{A}}((D_{m,2}))$ is the largest root of $\rho^{4}-m\rho^2+2m-6=0$.
It follows that $\rho_{\mathcal{A}}((D_{m,2}))=\sqrt{\frac{m+\sqrt{m^2-8m+24}}{2}}$.
Thus by a result in \cite{ZSWB} on the relationship between the eigenvalues of the adjacency tensors of a hypergraph and its $k$th power, we have $\rho_{\mathcal{A}}((D^k_{m,2}))=\sqrt[k]{\frac{m+\sqrt{m^2-8m+24}}{2}}$.

Let $v_1v_2v_3v_4$ be a path of $D_{m,1}$ such that $v_2$ is of degree $m-1$ and $v_3$ is of degree $2$.
Let $\mathbf{y}$ be the $2$-unit positive eigenvector of $\mathcal{ABC}(D_{m,1})$ corresponding to $\rho=\rho_{\mathcal{ABC}}(D_{m,1})$.
Let $y_i=y_{v_i}$ for $i=1,\ldots,4$.
By Lemma~\ref{ABC-permutation} and the $(\rho, \mathbf{y})$-eigenequations of $\mathcal{ABC}(D_{m,1})$, we have
\begin{align*}
\rho y_1&=\sqrt{\frac{m-2}{m-1}}y_2,\\
\rho y_2&=(m-2)\sqrt{\frac{m-2}{m-1}}y_1+ \sqrt{\frac{1}{2}}y_3,\\
\rho y_3&=\sqrt{\frac{1}{2}}y_2+\sqrt{\frac{1}{2}}y_4,\\
\rho y_4&=\sqrt{\frac{1}{2}}y_3.
\end{align*}
Then $\rho_{\mathcal{ABC}}(D_{m,1})$ is the largest root of
\[2\rho^{4}-\frac{2(m^2-3m+3)}{m-1}\rho^2+\frac{(m-2)^2}{m-1}=0,\]
 which
implies that $\rho_{\mathcal{ABC}}(D_{m,1})=\sqrt{\frac{m^2-3m+3+\sqrt{(m-1)^2+(m-2)^4}}{2(m-1)}}$.
From Theorem~\ref{power}, $\rho_{\mathcal{ABC}}(D^k_{m,1})=\sqrt[k]{\frac{m^2-3m+3+\sqrt{(m-1)^2+(m-2)^4}}{2(m-1)}}$.

Now, we prove the result. It is trivial if $m=3$. Suppose that $m\geq 4$.
Let $G$ be a $k$-uniform hypertree different from $S_{m,k}$ of size $m$.
Suppose that $G$ is different from  $D_{m,1}^k$. Then  the maximum degree of $G$ is at most $m-2$.
So, by Theorem \ref{med} and Lemma \ref{tree}(i), we have
\begin{align*}
\rho_{\mathcal{ABC}}(G)&\leq \sqrt[k]{\frac{m-3}{m-2}}\rho_{\mathcal{A}}(G)\\
&\leq \sqrt[k]{\frac{m-3}{m-2}}\sqrt[k]{\frac{m+\sqrt{m^2-8m+24}}{2}}\\
&= \sqrt[k]{\frac{(m-3)(m+\sqrt{m^2-8m+24})}{2(m-2)}}.
\end{align*}
So, it suffices to show that
\[
\sqrt[k]{\frac{(m-3)(m+\sqrt{m^2-8m+24})}{2(m-2)}}
<\rho_{\mathcal{ABC}}(D_{m,1}^k),
\]
which is indeed true as
\begin{align*}
&\quad \rho_{\mathcal{ABC}}^k(D_{m,1}^k)-\frac{(m-3)(m+\sqrt{m^2-8m+24})}{2(m-2)}\\
&=\frac{m^2-3m+3+\sqrt{(m-1)^2+(m-2)^4}}{2(m-1)}-\frac{(m-3)(m+\sqrt{m^2-8m+24})}{2(m-2)}\\
&=\frac{1}{2(m-2)(m-1)}\left(\left(m^2-3m+3+\sqrt{(m-1)^2+(m-2)^4}\right)(m-2)\right.\\
&\quad \left.-(m-3)(m-1)\left(m+\sqrt{m^2-8m+24}\right)\right)\\
&>\frac{1}{2(m-2)(m-1)}\left(\vphantom{\sqrt{m^2}}\left(m^2-3m+3+(m-2)^2\right)(m-2)\right.\\
&\quad \left.-(m-3)(m-1)\left(m+\sqrt{m^2-8m+24}\right)\right)\\
&=\frac{1}{2(m-2)(m-1)}\left(m^3 - 7m^2 + 18m - 14 -(m-3)(m-1)\sqrt{m^2-8m+24}\right)\\
&=\frac{1}{2(m-2)(m-1)}\left((m-3)(m-1)\left(m-3-\sqrt{(m-4)^2+8}\right)+3m-5\right) \\
&>0
\end{align*}
for $m\ge 3$. This completes the proof.
\end{proof}

\begin{Lemma}\label{tree-3-unif+}
For $m\geq 4$ and $k\geq 3$, $\rho_{\mathcal{ABC}}(S_{m,k;m-3,1,1})=b_m^\frac{1}{k}$, where  $b_m$ is the largest root of $\eta(t)=0$,
and
\begin{equation}\label{Eq0803-1}
\eta(t)=t^{3}-\frac{4m^2-19m+27}{4(m-2)}t^{2}+\frac{4m^2-23m+34}{4(m-2)}t-\frac{(m-3)^2}{4(m-2)}.
\end{equation}
\end{Lemma}

\begin{proof}
Let $v_1e_1v_2e_2v_3e_3v_4$ be a path of $S_{m,3;m-3,1,1}$ such that $v_2$ is of degree $m-2$ and
$v_3$ is of degree $2$.
Let ${\bf x}$ be the $3$-unit positive eigenvector of $\mathcal{ABC}(S_{m,3;m-3,1,1})$ corresponding to $\rho=\rho_{\mathcal{ABC}}(S_{m,3;m-3,1,1})$.
Let $x_i=x_{v_i}$ for $i=1,2,3,4$.
By Lemma~\ref{ABC-permutation} and the $(\rho, {\bf x})$-eigenequations  of $\mathcal{ABC}(S_{m,3;m-3,1,1})$, we have
\begin{align}
\rho x_1^2&=\sqrt[3]{\frac{m-3}{m-2}}x_1x_2,\label{Eq0810-1-1}\\
\rho x_2^2&=(m-3)\sqrt[3]{\frac{m-3}{m-2}}x_1^2+ \sqrt[3]{\frac{m-1}{4(m-2)}}x_3^2,\label{Eq0810-1-2}\\
\rho x_3^2&=\sqrt[3]{\frac{m-1}{4(m-2)}}x_2x_3+\sqrt[3]{\frac{1}{2}}x_4^2,\label{Eq0810-1-3}\\
\rho x_4^2&=\sqrt[3]{\frac{1}{2}}x_3x_4.\label{Eq0810-1-4}
\end{align}
By \eqref{Eq0810-1-1} and \eqref{Eq0810-1-4}, we have $
x_1=\frac{\sqrt[3]{\frac{m-3}{m-2}}}{\rho}x_2$
and
$x_4=\frac{\sqrt[3]{\frac{1}{2}}}{\rho}x_3$.
Substituting the expression for $x_4$ in \eqref{Eq0810-1-3},
we have $x_3=\frac{\sqrt[3]{\frac{m-1}{4(m-2)}}\,\rho^2  x_2}{\rho^3-\frac{1}{2}}$.
Eliminating  $x_1$ and $x_3$ from \eqref{Eq0810-1-2}, we have
\[
\rho^3\left(\rho^3-\frac{1}{2}\right)^2- \frac{(m-3)^2}{m-2}\left(\rho^3-\frac{1}{2}\right)^2 -\frac{m-1}{4(m-2)}\rho^6=0,
\]
i.e.,
\[
\eta(\rho^3)=0.
\]
So $\rho_{\mathcal{ABC}}(S_{m,3;m-3,1,1})$ is the largest root of $\eta(t^3)=0$. Let $b_m$ be the largest root of $\eta(t)=0$. Then $\rho_{\mathcal{ABC}}(S_{m,3;m-3,1,1})=b_m^{\frac{1}{3}}$.
By  Theorem~\ref{power}, $\rho_{\mathcal{ABC}}(S_{m,k;m-3,1,1})=\rho_{\mathcal{ABC}}^{\frac{3}{k}}(S_{m,3;m-3,1,1})=b_m^{\frac{1}{k}}$.
\end{proof}

By Lemma~\ref{tree-3-unif+},
$\rho_{\mathcal{ABC}}(S_{m,3;m-3,1,1})$ is the largest root of $\eta(t^3)=0$,
where
$\eta(t)$ is given in \eqref{Eq0803-1}.

We introduce four $3$-uniform hypergraphs.
Let $T_{m,1}=S_{m,3;m-4,2,1}$ for $m\geq 6$.
For $m\geq 5$,
let $T_{m,2}$ be the $3$-uniform hypertree obtained by adding a pendant edge at a pendant vertex in a pendant edge at the vertex of degree $m-3$ in $S_{m-1,3;m-4,1,1}$,
$T_{m,3}$ be the $3$-uniform hypertree obtained by adding respectively one pendant edge at two distinct pendant vertices in a pendant edge at the vertex of degree $2$ in $D_{m-2,1}^3$ and
$T_{m,4}$ be the $3$-uniform hypertree obtained by adding a pendant edge at a pendant vertex in a pendant edge at the vertex of degree $2$ in $S_{m-1,3;m-4,1,1}$.
Obviously, $T_{5,2}\cong T_{5,3} \cong T_{5,4}$.

\begin{Lemma}\label{tree-3-unif} The following statements are true:

(i) For $m\geq 6$, $\max\{\rho_{\mathcal{ABC}}(T_{m,1}),\rho_{\mathcal{ABC}}(T_{m,2})\}<\rho_{\mathcal{ABC}}(S_{m,3;m-3,1,1})$.

(ii) For $m\geq 5$, $\max\{\rho_{\mathcal{ABC}}(T_{m,3}), \rho_{\mathcal{ABC}}(T_{m,4})\}<\rho_{\mathcal{ABC}}(S_{m,3;m-3,1,1})$.
\end{Lemma}

\begin{proof}
First, we prove (i).

Let $v_1e_1v_2e_2v_3e_3v_4$ be a path in $T_{m,1}$ such that $v_2$ and $v_3$ are the vertices of degrees
$m-3$ and $3$, respectively. Let $v_5$ is a vertex of degrees $2$ in $e_2$, and $v_6$ be a pendant vertex in the pendant edge at $v_5$.
Let $\rho_1=\rho_{\mathcal{ABC}}(T_{m,1})$.
Let ${\bf x}$ be the $3$-unit positive eigenvector of $\mathcal{ABC}(T_{m,1})$ corresponding to $\rho_1$.
Let $x_i=x_{v_i}$ for $i=1,\ldots,6$.
By Lemma~\ref{ABC-permutation} and the $(\rho_1, {\bf x})$-eigenequations of $\mathcal{ABC}(T_{m,1})$, we have
\begin{align}
\rho_1 x_1^{2}&=\sqrt[3]{\frac{m-4}{m-3}}x_1x_2,\label{Eq0810-2-1}\\
\rho_1 x_2^{2}&=(m-4)\sqrt[3]{\frac{m-4}{m-3}}x_1^2+\sqrt[3]{\frac{m-1}{6(m-3)}}x_3x_5,\label{Eq0810-2-2}\\
\rho_1 x_3^{2}&=\sqrt[3]{\frac{m-1}{6(m-3)}}x_2x_5+2\sqrt[3]{\frac{2}{3}}x_4^2,\label{Eq0810-2-3}\\
\rho_1 x_4^{2}&=\sqrt[3]{\frac{2}{3}}x_3x_4,\label{Eq0810-2-4}\\
\rho_1 x_5^{2}&=\sqrt[3]{\frac{m-1}{6(m-3)}}x_2x_3+\sqrt[3]{\frac{1}{2}}x_6^2,\label{Eq0810-2-5}\\
\rho_1 x_6^{2}&=\sqrt[3]{\frac{1}{2}}x_5x_6.\label{Eq0810-2-6}
\end{align}
By \eqref{Eq0810-2-1}, we have $x_1=\frac{\sqrt[3]{\frac{m-4}{m-3}}}{\rho_1}x_2$.
Combining \eqref{Eq0810-2-4} and \eqref{Eq0810-2-6} yields $x_4=\frac{\sqrt[3]{\frac{2}{3}}x_3}{\rho_1}$
and $x_6=\frac{\sqrt[3]{\frac{1}{2}}x_5}{\rho_1}$.
Eliminating $x_4$ and $x_6$ from \eqref{Eq0810-2-3} and \eqref{Eq0810-2-5}, respectively,
we have
\[
\left(\rho_1^3-\frac{4}{3}\right)x^2_3=\sqrt[3]{\frac{m-1}{6(m-3)}}x_2x_5\rho_1^2
\]
and
\[
\left(\rho_1^3-\frac{1}{2}\right)x^2_5=\sqrt[3]{\frac{m-1}{6(m-3)}}x_2x_3\rho_1^2,
\]
so
$x_3=\frac{\sqrt[3]{\frac{m-1}{6(m-3)}}\rho_1^2}{\left(\rho_1^3-\frac{1}{2}\right)^\frac{1}{3}\left(\rho_1^3-\frac{4}{3}\right)^\frac{2}{3}}x_2$
and
$x_5=\frac{\sqrt[3]{\frac{m-1}{6(m-3)}}\rho_1^2}{\left(\rho_1^3-\frac{1}{2}\right)^\frac{2}{3}\left(\rho_1^3-\frac{4}{3}\right)^\frac{1}{3}}x_2 $.
Eliminating $x_1$, $x_3$ and $x_5$ from \eqref{Eq0810-2-2}, we have
\[
h_1(\rho_1^3)=0,
\]
where
\[
h_1(t)=t^3-\frac{3m^2- 18m+ 31}{3(m-3)}t^2+\frac{11m^2-84m+ 164}{6(m-3)}t-\frac{ 2m^2- 16m + 32}{3(m-3)}.
\]
So $\rho_1$ is the largest root of $h_1(t^3)=0$.

Bearing in mind the expression for $\eta(t)$ in \eqref{Eq0803-1}, we have
\[
\eta(t)=h_1(t)-\frac{(m - 5)p_1(t)}{12(m - 2)(m - 3)},
\]
where
\[
p_1(t)=(3m-1)t^2 +(10m^2- 57m+ 70)t - 5m^2 + 28m - 35.
\]
Let $t_1$ be the largest root of $h_1(t)$.

If $m=6$, then $h_1(t)=\frac{9t^3-31t^2+28t-8}{9}$.
It is easy to seen that  $h_1^{(1)}(3)=\frac{85}{9}$,
$h_1^{(2)}(3)=\frac{100}{9}$ and $h_1^{(3)}(3)=6$.
Since $h_1^{(3-i)}(t)$ is strictly increasing for $t\geq 3$ as $h_1^{(4-i)}(3)>0$ with $i=1,2,3$,
$h_1(t)$ is strictly increasing for $t\geq 3$.
Noting that $h_1(2)=-\frac{4}{9}$ and $h_1(3)=\frac{40}{9}$,  $t_1$ lies in $ (2,3)$.
As $p_1(t)$ is increasing for $t\in [2,3]$ and $p_1(t)>p_1(2)=13$,
we have $\eta(t_1)=h_1(t_1)-\frac{p_1(t_1)}{144}<-\frac{p_1(2)}{144}<0$.
So $t_1$ is less than the largest root of $\eta(t)=0$.

Suppose that $m\geq 7$.
Note that
\begin{align*}
h_1^{(1)} (t)&=3t^2-\frac{6m^2 - 36m + 62}{3(m - 3)}t +\frac{11m^2-84m+164}{6(m - 3)},\\
h_1^{(2)} (t)&=6t-\frac{6m^2 - 36m + 62}{3(m - 3)},\\
h_1^{(3)} (t)&=6.
\end{align*}
Then
\begin{align*}
h_1^{(1)} (m-4)&=\frac{6m^3 - 67m^2 + 224m - 204}{6(m - 3)}>0,\\
h_1^{(2)} (m-4)&=\frac{2(6m^2 - 45m + 77)}{3(m - 3)}>0,
\end{align*}
and $h_1^{(2)} (t)$ is strictly increasing for $t\geq m-4$.
Therefore  $h_1^{(1)}(t)$ is strictly increasing for $t\geq m-4$
and thus
$h_1(t)$ is increasing for $t\geq m-4$.
Note  that
\begin{align*}
h_1 (m-5)&=-\frac{m^3 - 5m^2 - 36m + 184}{6(m - 3)}<0,\\
h_1 (m-4)&=\frac{(m - 4)(5m^2 - 54m + 140)}{6(m - 3)}>0.
\end{align*}
So $t_1$ lies in $(m-5,m-4)$.
It is easy to see that $p_1(t)$ is strictly increasing for $t\in [m-5,m-4]$.
So
\[
p_1(t)> p_1(m-5)=13m^3 - 143m^2 + 468m - 410>0
\]
for $t\in (m-5,m-4)$.
Then 
\[\eta(t_1)=h_1(t_1)-\frac{(m - 5)p_1(t_1)}{12(m - 2)(m - 3)}<0.\]
So $t_1$ is less than the largest root of $\eta(t)=0$.
By Lemma~\ref{tree-3-unif+},
$\rho_{\mathcal{ABC}}(S_{m,3;m-3,1,1})$ is the largest root of $\eta(t^3)=0$.
Thus
$\rho_1^3=t_1<\rho^3(S_{m,3;m-3,1,1})$. That is,
$\rho_{\mathcal{ABC}}(T_{m,1})<\rho_{\mathcal{ABC}}(S_{m,3;m-3,1,1})$.

Let $v_1e_1v_2e_2v_3e_3v_4e_4v_5$ be a path in $T_{m,2}$ such that there is a pendant edge at a vertex in $e_3\setminus\{v_3,v_4\}$. Let $v_6$ and $v_7$ be  pendant vertices in $e_2$ and a pendant edge at $v_3$, respectively.
Let $\rho_2=\rho_{\mathcal{ABC}}(T_{m,2})$.
Let ${\bf x}$ be the $3$-unit positive eigenvector of $\mathcal{ABC}(T_{m,2})$ corresponding to $\rho_2$.
Let $x_i=x_{v_i}$ for $i=1,\ldots,7$.
By Lemma~\ref{ABC-permutation} and the $(\rho_2, {\bf x})$-eigenequations of $\mathcal{ABC}(T_{m,2})$, we have
\begin{align}
\rho_2 x_1^{2}&=\sqrt[3]{\frac{1}{2}}x_1x_2,\label{Eq0810-3-1}\\
\rho_2 x_2^{2}&=\sqrt[3]{\frac{1}{2}}x_1^2+\sqrt[3]{\frac{1}{2}}x_3x_6,\label{Eq0810-3-2}\\
\rho_2 x_3^{2}&=\sqrt[3]{\frac{1}{2}}x_2x_6+\sqrt[3]{\frac{m-2}{4(m-3)}}x_4^2
+(m-5)\sqrt[3]{\frac{m-4}{m-3}}x_7^2,\label{Eq0810-3-3}\\
\rho_2 x_4^{2}&=\sqrt[3]{\frac{m-2}{4(m-3)}}x_3x_4+\sqrt[3]{\frac{1}{2}}x_5^2,\label{Eq0810-3-4}\\
\rho_2 x_5^{2}&=\sqrt[3]{\frac{1}{2}}x_4x_5,\label{Eq0810-3-5}\\
\rho_2 x_6^{2}&=\sqrt[3]{\frac{1}{2}}x_2x_3,\label{Eq0810-3-6}\\
\rho_2 x_7^{2}&=\sqrt[3]{\frac{m-4}{m-3}}x_3x_7.\label{Eq0810-3-7}
\end{align}
By \eqref{Eq0810-3-1}, \eqref{Eq0810-3-2} and \eqref{Eq0810-3-6}, we have $x_2=\frac{\sqrt[3]{\frac{1}{2}}\rho_2}{\left(\rho_2^3-\frac{1}{2}\right)^\frac{2}{3}}x_3$ and
$x_6=\frac{\sqrt[3]{\frac{1}{2}}}{\left(\rho_2^3-\frac{1}{2}\right)^\frac{1}{3}}x_3$.
By \eqref{Eq0810-3-4} and \eqref{Eq0810-3-5}, we have
$x_4=\frac{\sqrt[3]{\frac{m-2}{4(m-3)}}\rho_2^2}{\rho_2^3-\frac{1}{2}}x_3$.
By \eqref{Eq0810-3-7}, it is easily seen that $x_7=\frac{\sqrt[3]{\frac{m-4}{m-3}}}{\rho_2}x_3$.
Eliminating $x_2, x_4, x_6$ and $x_7$ from \eqref{Eq0810-3-3},
it follows that
\[
h_2(\rho_2^3)=0,
\]
where
\[
h_2(t)=t^3-\frac{4m^2-29m+ 60}{4(m-3)}t^2+\frac{2m^2-17m+37}{2(m-3)}t-\frac{m^2-9m + 20}{4(m-3)}.
\]
So $\rho_2$ is the largest root of $h_2(t^3)=0$.

It is easily seen that
\[
\eta(t)=h_2(t)-\frac{p_2(t)}{4(m - 2)(m - 3)},
\]
where
\[
p_2(t)=(6m^2- 34m+ 39)t^2 - (7m^2- 39m+ 46)t + 2m^2 - 11m + 13.
\]
As
\begin{align*}
h_2^{(1)} (t)&=3t^2-\frac{4m^2-29m+ 60}{2(m - 3)}t +\frac{2m^2-17m+37}{2(m - 3)},\\
h_2^{(2)} (t)&=6t-\frac{4m^2-29m+ 60}{2(m - 3)},\\
h_2^{(3)} (t)&=6,
\end{align*}
we have
\begin{align*}
h_2^{(1)} (m-4)&=\frac{2m^3 - 19m^2 + 47m - 11}{2(m - 3)}>0,\\
h_2^{(2)} (m-4)&=\frac{8m^2 - 55m + 84}{2(m - 3)}>0.
\end{align*}
Since $h_2^{(3-i)}(t)$ is strictly increasing for $t\geq m-4$ as $h_2^{(4-i)}(m-4)>0$ with $i=1,2,3$,
$h_2(t)$ is strictly increasing for $t\geq m-4$.

Note  that
\begin{align*}
h_2 (m-5)&=-\frac{3m^3 - 37m^2 + 109m + 32}{4(m - 3)}<0 ~\text{if}~m=6,7,8,\\
h_2 (m-6)&=-\frac{3m^3 - 37m^2 + 109m + 32}{4(m - 3)}<0~\text{if}~m\geq 9,\\
h_2 (m-4)&=\frac{(m - 4)(5m^2 - 51m + 127)}{4(m - 3)}>0.
\end{align*}
Then the largest root of $h_2(t)=0$, say $t_2$ lies in $(m-5,m-4)$ if m=6,7,8 and lies in $(m-6,m-4)$ if $m\geq 9$.
It is easy to see that $p_2(t)$ is strictly increasing for $t\in [m-6,m-4]$.
So
\[
p_2(t)> p_2(m-6)=6m^4 - 113m^3 + 746m^2 - 1983m + 1693>0
\]
for $t\in (m-6,m-4)$.
Then
\[\eta(t_2)=h_2(t_2)-\frac{p_2(t_2)}{4(m - 2)(m - 3)}<-\frac{p_2(m-6)}{4(m - 2)(m - 3)}<0.\]
So  $t_2$ is less than the largest root of $\eta(t)=0$.
By Lemma \ref{tree-3-unif+}, we have  $\rho_{\mathcal{ABC}}(T_{m,2})<\rho_{\mathcal{ABC}}(S_{m,3;m-3,1,1})$.

Now, we prove (ii).

Let $v_1e_1v_2e_2v_3e_3v_4e_4v_5$ be a path in $T_{m,3}$ such that $v_2$ is the vertex of degree $m-3$. Let $v_6$ be a pendant vertex in $e_2$.
Let $\rho_3=\rho_{\mathcal{ABC}}(T_{m,3})$.
Let ${\bf x}$ be the $3$-unit positive eigenvector of $\mathcal{ABC}(T_{m,3})$ corresponding to $\rho_3$.
Let $x_i=x_{v_i}$ for $i=1,\ldots,6$.
By Lemma~\ref{ABC-permutation} and the $(\rho_3, {\bf x})$-eigenequations of $\mathcal{ABC}(T_{m,3})$, we have
\begin{align}
\rho_3 x_1^{2}&=\sqrt[3]{\frac{m-4}{m-3}}x_1x_2,\label{Eq0810-6-1}\\
\rho_3 x_2^{2}&=(m-4)\sqrt[3]{\frac{m-4}{m-3}}x_1^2+\sqrt[3]{\frac{1}{2}}x_3x_6,\label{Eq0810-6-2}\\
\rho_3x_3^{2}&=\sqrt[3]{\frac{1}{2}}x_2x_6+\sqrt[3]{\frac{3}{8}}x_4^2,\label{Eq0810-6-3}\\
\rho_3 x_4^{2}&=\sqrt[3]{\frac{3}{8}}x_3x_4+\sqrt[3]{\frac{1}{2}}x_5^2,\label{Eq0810-6-4}\\
\rho_3 x_5^{2}&=\sqrt[3]{\frac{1}{2}}x_4x_5,\label{Eq0810-6-5}\\
\rho_3 x_6^{2}&=\sqrt[3]{\frac{1}{2}}x_2x_3.\label{Eq0810-6-6}
\end{align}
By \eqref{Eq0810-6-4} and \eqref{Eq0810-6-5}, we have $x_4=\frac{\sqrt[3]{\frac{3}{8}}\rho_3^3}{\rho_3^3-\frac{1}{2}}x_3$, which together with \eqref{Eq0810-6-3} and \eqref{Eq0810-6-6},  implies $x_3=\frac{\sqrt[3]{\frac{1}{2}}
\left(\rho_3^3-\frac{1}{2}\right)^\frac{4}{3}}{\rho_3
\left(\left(\rho_3^3-\frac{1}{2}\right)^3-\frac{3}{8}\rho_3^3\right)^\frac{2}{3}}x_2$
and
$x_6=\frac{\sqrt[3]{\frac{1}{2}}\left(\rho_3^3-\frac{1}{2}\right)^\frac{2}{3}}{\rho_3\left(\left(\rho_3^3-\frac{1}{2}\right)^2-\frac{3}{8}\rho_3^3\right)^\frac{1}{3}}x_2$.
Eliminating $x_1$, $x_3$ and $x_6$ from \eqref{Eq0810-6-2}, we have
\[
h_3(t^3)=0,
\]
where
\[
h_3(t)=t^3-\frac{8m^2-49m+ 83}{8(m-3)}t^2+\frac{11m^2-82m+158}{8(m-3)}t-\frac{2m^2-15m + 29}{8(m-3)}.
\]
So $\rho_3$ is the largest root of $h_3(t^3)=0$.

Note that
\[
\eta(t)=h_3(t)-\frac{(m-4)p_3(t)}{8(m - 2)(m - 3)},
\]
where
\[
p_3(t)=(3m-1)t^2 +(3m^2 - 22m+28)t + m - 1.
\]
Let $t_3$ be the maximum root of $h_3(t)=0$.

If $m=5$, then $t_3=\frac{15+\sqrt{97}}{16}$ as $h_3(t)=\frac{(2t-1)(8t^2-15t+4)}{16}$,
and so $\eta(t_3)=-\frac{77(15+\sqrt{97})}{3072}+\frac{1}{16}\approx -0.5603$.
If $m=6$, then $h_3(t)=\frac{24t^3-77t^2+62t-11}{24}$.
Note that $h_3(t)$ is strictly increasing for $t\geq 3$, because
$h_3^{(3-i)}(t)$ is strictly increasing for $t\geq 3$
as $h_3^{(4-i)}(3)>0$ with $i=1,2,3$.
As $h_3(2)<0$ and $h_3(3)>0$,  $t_3$ lies in $ (2,3)$.
Since $p_3(t)$ is increasing for $t\in [2,3]$ and $p_3(2)>0$,
$\eta(t_3)=h_3(t_3)-\frac{p_3(t_3)}{48}<-\frac{p_3(2)}{48}<0$.
So $t_3$ is less than the largest root of $\eta(t)=0$ for $m=5,6$.

Suppose that $m\geq 7$.
As
\begin{align*}
h_3^{(1)} (t)&=3t^2-\frac{8m^2-49m+ 83}{4(m - 3)}t +\frac{11m^2-82m+158}{8(m - 3)},\\
h_3^{(2)} (t)&=6t-\frac{8m^2-49m+ 83}{4(m - 3)},\\
h_3^{(3)} (t)&=6,
\end{align*}
we have
\begin{align*}
h_3^{(1)} (m-4)&=\frac{8m^3 - 91m^2 + 320m - 330}{8(m - 3)}>0,\\
h_3^{(2)} (m-4)&=\frac{16m^2 -119m + 205}{4(m - 3)}>0.
\end{align*}
Since $h_3^{(3-i)}(t)$ is strictly increasing for $t\geq m-4$ as $h_1^{(4-i)}(m-4)>0$ with $i=1,2,3$,
$h_3(t)$ is strictly increasing for $t\geq m-4$.
Note  that
\begin{align*}
h_3 (m-5)&=-\frac{2m^3 - 24m^2 + 81m - 53}{4(m - 3)}<0,\\
h_3 (m-4)&=\frac{4m^3 - 59m^2 + 285m - 453}{8(m - 3)}>0.
\end{align*}
So $t_3$ lies in $(m-5,m-4)$.
As $p_3(t)$ is strictly increasing for $t\in [m-5,m-4]$,
we have
\[
p_3(t)\geq p_3(m-5)=2(3m^3 - 34m^2 + 112m - 83)>0
\]
for $t\in (m-5,m-4)$.
Then
\[\eta(t_3)=h_3(t_3)-\frac{(m - 4)p_3(t_3)}{8(m - 2)(m - 3)}<0.\]
So $t_3$ is less than the largest root of $\eta(t)=0$.
Thus $\rho_{\mathcal{ABC}}(T_{m,3})<\rho_{\mathcal{ABC}}(S_{m,3;m-3,1,1})$.

As $T_{5,3}\cong T_{5,4}$ and $\rho_{\mathcal{ABC}}(T_{5,3})<\rho_{\mathcal{ABC}}(S_{5,3;2,1,1})$, it is sufficient to consider $T_{m,4}$ for $m\geq 6$.
Let $v_1e_1v_2e_2v_3e_3v_4e_4v_5$ be a path in $T_{m,4}$ such that $v_2$ is the vertex of degree $m-3$. Let $v_6$ be the vertex in $e_2\setminus\{v_2,v_3\}$, and $v_7$ and $v_8$ be the pendant vertices
in $e_3$ and a pendant edge at $v_6$, respectively.
Let $\rho_4=\rho_{\mathcal{ABC}}(T_{m,4})$.
Let ${\bf x}$ be the $3$-unit positive eigenvector of $\mathcal{ABC}(T_{m,4})$ corresponding to $\rho_4$.
Let $x_i=x_{v_i}$ for $i=1,\ldots,8$.
By Lemma~\ref{ABC-permutation} and the $(\rho_4, {\bf x})$-eigenequations of $\mathcal{ABC}(T_{m,4})$, we have
\begin{align}
\rho_4 x_1^{2}&=\sqrt[3]{\frac{m-4}{m-3}}x_1x_2,\label{Eq0810-4-1}\\
\rho_4 x_2^{2}&=(m-4)\sqrt[3]{\frac{m-4}{m-3}}x_1^2+\sqrt[3]{\frac{m-2}{4(m-3)}}x_3x_6,\label{Eq0810-4-2}\\
\rho_4 x_3^{2}&=\sqrt[3]{\frac{m-2}{4(m-3)}}x_2x_6+\sqrt[3]{\frac{1}{2}}x_4x_7,\label{Eq0810-4-3}\\
\rho_4 x_4^{2}&=\sqrt[3]{\frac{1}{2}}x_3x_7+\sqrt[3]{\frac{1}{2}}x_5^2,\label{Eq0810-4-4}\\
\rho_4 x_5^{2}&=\sqrt[3]{\frac{1}{2}}x_4x_5,\label{Eq0810-4-5}\\
\rho_4 x_6^{2}&=\sqrt[3]{\frac{m-2}{4(m-3)}}x_2x_3+\sqrt[3]{\frac{1}{2}}x_8^2,\label{Eq0810-4-6}\\
\rho_4 x_7^{2}&=\sqrt[3]{\frac{1}{2}}x_3x_4,\label{Eq0810-4-7}\\
\rho_4 x_8^{2}&=\sqrt[3]{\frac{1}{2}}x_6x_8.\label{Eq0810-4-8}
\end{align}
From \eqref{Eq0810-4-4}, \eqref{Eq0810-4-5} and \eqref{Eq0810-4-7},
we get $x_3=\frac{(\rho_4^3-\frac{1}{2})^\frac{2}{3} }{\sqrt[3]{\frac{1}{2}}\rho_4}x_4$.
From \eqref{Eq0810-4-1} and \eqref{Eq0810-4-2}, we have
\begin{equation}\label{Eq0803-3}
\left(\rho_4^3-\frac{(m-4)^2}{m-3}\right)x_2^2=\sqrt[3]{\frac{m-2}{4(m-3)}}\rho_4^2x_3x_6.
\end{equation}
From \eqref{Eq0810-4-6}  and \eqref{Eq0810-4-8}, we have
\begin{equation}\label{Eq0803-4}
\left(\rho_4^3-\frac{1}{2}\right)x_6^2=\sqrt[3]{\frac{m-2}{4(m-3)}}\rho_4^2x_2x_3.
\end{equation}
By \eqref{Eq0803-3} and \eqref{Eq0803-4}, we have
\[
x_2=
\frac{\sqrt[3]{\frac{m-2}{4(m-3)}}\rho_4^2}{\left(\rho_4^3-\frac{(m-4)^2}{m-3}\right)^\frac{2}{3}
\left(\rho_4^3-\frac{1}{2}\right)^\frac{1}{3}}x_3,
\]
so
\[
x_2=\frac{\sqrt[3]{\frac{m-2}{2(m-3)}}\left(\rho_4^3-\frac{1}{2}\right)^\frac{1}{3}\rho_4}
{\left(\rho_4^3-\frac{(m-4)^2}{m-3}\right)^\frac{2}{3}}x_4.
\]
By \eqref{Eq0810-4-3}  and \eqref{Eq0810-4-7} , we have
\begin{equation}\label{Eq0803-6}
\rho_4^3x_3^3-\rho_4^3x_7^3=\sqrt[3]{\frac{m-2}{4(m-3)}}\rho_4^2x_2x_3x_6,
\end{equation}
from which, by combining \eqref{Eq0810-4-4}, \eqref{Eq0810-4-5} and \eqref{Eq0810-4-7}, we have
\begin{equation}\label{Eq0803-7}
\rho_4^3x_7^3=\left(\rho_4^3-\frac{1}{2}\right)x_4^3.
\end{equation}
Now by \eqref{Eq0803-3}, \eqref{Eq0803-6} and \eqref{Eq0803-7}, we have
\begin{equation}\label{Eq0803-2}
\left(\rho_4^3-\frac{(m-4)^2}{m-3}\right)x_2^3=\rho_4^3x_3^3-\left(\rho_4^3-\frac{1}{2}\right)x_4^3.
\end{equation}
Eliminating $x_2$ and $x_3$ from \eqref{Eq0803-2},
it follows that
\[
h_4(\rho_4^3)=0,
\]
where
\[h_4(t)=\frac{1}{8(m - 3)}(2t- 1)(4(m -3)t^2 -(4m^2-27m +50)t+4m^2 -32m +64).\]
So $\rho_4$ is the largest root of $h_4(t^3)=0$.

Note that
\[
\eta(t)=h_4(t)-\frac{p_4(t)}{8(m - 2)(m - 3)},
\]
where
\begin{align*}
p_4(t)&=(4m^2- 20m+ 14)t^2 +(4m^3 - 45m^2+ 154m - 152)t\\
&\quad  - 2m^3 + 22m^2 - 74m + 74.
\end{align*}
Let $t_4$ be the largest root of $h_4(t)=0$.

If $m=6$, then $t_4=2$ since $h_4(t)=\frac{(2t-1)(3t-2)(t-2)}{6}$,
and so $\eta(t_4)=-\frac{29}{16}$. Thus $t_4$ is less than the largest root of $h_4(t)=0$.

Suppose that $m\geq 7$.
As
\begin{align*}
h_4^{(1)} (t)&=3t^2+\frac{ -16m^2t+100mt-176t +12m^2 - 91m + 178}{8(m - 3)},\\
h_4^{(2)} (t)&=6t+\frac{-4m^2+25m -44}{2(m - 3)},\\
h_4^{(3)} (t)&=6,
\end{align*}
we have
\begin{align*}
h_4^{(1)} (m-4)&=\frac{8m^3 - 88m^2 + 293m - 270}{8(m - 3)}>0,\\
h_4^{(2)} (m-4)&=\frac{8m^2 - 59m + 100}{2(m - 3)}>0.
\end{align*}
Since $h_4^{(3-i)}(t)$ is strictly increasing for $t\geq m-5$ as $h_4^{(4-i)}(m-5)>0$ with $i=1,2,3$,
$h_4(t)$ is strictly increasing for $t\geq m-5$.
Note  that
\begin{align*}
h_4 (m-5)&=-\frac{(2m - 11)(m^2 - 3m - 14)}{8(m - 3)}<0,\\
h_4 (m-4)&=\frac{3(m - 6)(2m - 9)(m - 4)}{8(m - 3)}>0.
\end{align*}
So $t_4$ lies in $(m-5,m-4)$.
As  $p_4(t)$ is strictly increasing for $t\in [m-5,m-4]$,
we have
\[
p_4(t)> p_4(m-5)=(m - 4)(8m^3 - 95m^2 + 335m - 296)>0
\]
for $t\in (m-5,m-4)$.
Then
\[\eta(t_4)=h_4(t_4)-\frac{p_4(t_4)}{8(m - 2)(m - 3)}<0.\]
So $t_4$ is less than the largest root of $\eta(t)=0$.
Thus $\rho_{\mathcal{ABC}}(T_{m,4})<\rho_{\mathcal{ABC}}(S_{m,3;m-3,1,1})$.
\end{proof}

\begin{Lemma}\label{tree-3-unif++}  For $m\geq 5$,  $\rho_{\mathcal{ABC}}(S_{m,4;m-4,1,1,1})< \rho_{\mathcal{ABC}}(S_{m,4;m-3,1,1})$.
\end{Lemma}

\begin{proof}
Let $v_1e_1v_2e_2v_3e_3v_4$ be a path in $S_{m,4;m-4,1,1,1}$ such that $v_2$ and $v_3$ are the vertices of degree
$m-3$ and $2$, respectively.
Let $\rho =\rho_{\mathcal{ABC}}(S_{m,4;m-4,1,1,1})$.
Let ${\bf x}$ be the $4$-unit positive eigenvector of $\mathcal{ABC}(S_{m,4;m-4,1,1,1})$ corresponding to $\rho $.
Let $x_i=x_{v_i}$ for $i=1,\ldots,4$.
By Lemma~\ref{ABC-permutation} and the $(\rho, {\bf x})$-eigenequations of $\mathcal{ABC}(S_{m,4;m-4,1,1,1})$, we have
\begin{align}
\rho  x_1^{3}&=\sqrt[4]{\frac{m-4}{m-3}}x_1^2x_2,\label{Eq0810-5-1}\\
\rho  x_2^{3}&=(m-4)\sqrt[4]{\frac{m-4}{m-3}}x_1^3+\sqrt[4]{\frac{m-1}{8(m-3)}}x_3^3,\label{Eq0810-5-2}\\
\rho  x_3^{3}&=\sqrt[4]{\frac{m-1}{8(m-3)}}x_2x_3^2+\sqrt[4]{\frac{1}{2}}x_4^3,\label{Eq0810-5-3}\\
\rho  x_4^{3}&=\sqrt[4]{\frac{1}{2}}x_3x_4^2.\label{Eq0810-5-4}
\end{align}
By \eqref{Eq0810-5-1}, we get $x_1=\frac{\sqrt[4]{\frac{m-4}{m-3}}}{\rho}x_2$.
By \eqref{Eq0810-5-3} and \eqref{Eq0810-5-4}, we have $x_3=\frac{\sqrt[4]{\frac{m-1}{8(m-3)}}\rho^3}{\rho^4-\frac{1}{2}}x_2$.
Eliminating  $x_1$ and $x_3$ from \eqref{Eq0810-5-2}, we
have \[ \frac{h(\rho^4)}{\rho^3\left(\rho^4-\frac{1}{2}\right)^3}=0,\]
where
\begin{align*}
h(t)&=t^4-\frac{8m^2- 51m+ 91}{8(m-3)}t^3+\frac{ 6m^2-45m+87}{4(m-3)}t^2\\
&\quad -\frac{6m^2-47m+93}{8(m-3)}t+\frac{m^2-8m +16}{8(m - 3)}.
\end{align*}
So $\rho$ is the largest root of $h(t^4)=0$.
By the expression for $\eta(t)$ given in \eqref{Eq0803-1}, we have
\[
t\eta(t)=h(t)-\frac{(m - 4)p(t)}{8(m - 2)(m - 3)},
\]
where
\[
p(t)=5(m-1)t^3+4(m^2-7m+9)t^2-(4m^2-25m+33)t+m^2-6m + 8.
\]
Let $t_0$ be the largest root of $h(t)=0$.

If $m=5$, then $h(t)=\frac{(4t-1)(4t^3-8t^2+4t-1)}{16}$.
Note that $h(t)$ is strictly increasing for $t\geq 2$, because
$h^{(4-i)}(t)$ is strictly increasing for $t\geq 2$
as $h^{(5-i)}(2)>0$ with $i=1,2,3,4$.
As $h(1)<0$ and $h(2)>0$,  $t_0$ lies in $(1,2)$.
Since $p(t)$ is increasing for $t\in [1,2]$ and $p(1)>0$,
$t_0\eta(t_0)=h(t_0)-\frac{p(t_0)}{48}<-\frac{p(1)}{48}<0$.
So $t_0$ is less than the largest root of $\eta(t)=0$.

Suppose that $m\geq 6$.
As
\begin{align*}
h^{(1)} (t)&=4t^3-\frac{24m^2 - 153m + 273}{8(m - 3)}t^2 +\frac{6m^2-45m+87}{2(m - 3)}t \\
&\quad -\frac{6m^2 - 47m+ 93}{8(m - 3)},\\
h^{(2)} (t)&=12t^2-\frac{24m^2 - 153m + 273}{4(m - 3)}t+\frac{6m^2-45m+87}{2(m - 3)},\\
h^{(3)} (t)&=24t-\frac{24m^2 - 153m + 273}{4(m - 3)},\\
h^{(4)} (t)&=24,
\end{align*}
we have
\begin{align*}
h^{(1)} (m-4)&=\frac{8m^4 - 111m^3 + 525m^2 - 909m + 291}{8(m - 3)}>0,\\
h^{(2)} (m-4)&=\frac{3(8m^3 - 89m^2 + 315m - 346)}{4(m - 3)}>0,\\
h^{(3)} (m-4)&=\frac{3(24m^2 - 173m + 293)}{4(m - 3)}>0.
\end{align*}
Since $h^{(4-i)}(t)$ is strictly increasing for $t\geq m-4$ as $h^{(5-i)}(m-4)>0$ with $i=1,2,3,4$,
$h(t)$ is strictly increasing for $t\geq m-4$.
Note that
\begin{align*}
h (m-5)&=-\frac{m^4 - 8m^3 - 42m^2 + 526m - 1206}{8(m - 3)}<0,\\
h (m-4)&=\frac{(m - 4)(7m^3 - 99m^2 + 462m - 713)}{8(m - 3)}>0,\\
\end{align*}
the largest root of $h(t)=0$ is in $(m-5,m-4)$.

Noting $p(t)$ is increasing for $t\in [m-5,m-4]$,
we have $p(t)\geq p(m-5)=9m^4 - 152m^3 + 912m^2 - 2224m + 1698>0$ for $t\in (m-5,m-4)$.
Then recalling the largest root, say $t_0$ of $h(t)=0$ is in $(m-5,m-4)$,
we get
\[t_0\eta(t_0)=h(t_0)-\frac{(m - 4)p(t_0)}{8(m - 2)(m - 3)}<0.\]
So $t_0$ is less than the largest root of $\eta(t)=0$. By Lemma~\ref{tree-3-unif+},
$\rho_{\mathcal{ABC}}(S_{m,4;m-3,1,1})$ is the largest root of $\eta(t^4)=0$.
Thus $\rho_{\mathcal{ABC}}(S_{m,4;m-4,1,1,1})<\rho_{\mathcal{ABC}}(S_{m,4;m-3,1,1})$.
\end{proof}

\begin{proof}[{\bf Proof of Theorem \ref{abc7-180-01}}]
By Lemma~\ref{tree-3-unif+}, we only need to show that
\[
\rho_{\mathcal{ABC}}(G)\leq b_m^\frac{1}{k}
\]
with equality if and only if $G\cong S_{m,k;m-3,1,1}$.
It is trivial for $m=4$.

Suppose that $m\ge 5$.
Let $G$ be a non-power $k$-uniform hypertree of size $m$ that maximizes the ABC spectral radius. Then we only need to show that $G\cong S_{m,k;m-3,1,1}$.

If $m=5$, then $G\cong S_{5,k;2,1,1}$, $T_{5,3}^{k}$ when $k=3$,
and $G\cong S_{5,k;2,1,1}$, $S_{5,k;1,1,1,1}$, $T_{5,3}^{k}$ when $k\geq 4$.
If $k\geq 3$, then by Theorem~\ref{power} and Lemma~\ref{tree-3-unif} (ii),
\[\rho_{\mathcal{ABC}}(T_{5,3}^{k})=\rho_{\mathcal{ABC}}^{\frac{3}{k}}(T_{5,3})<\rho_{\mathcal{ABC}}^{\frac{3}{k}}(S_{5,3;2,1,1})
=\rho_{\mathcal{ABC}}(S_{5,k;2,1,1})\]
and if $k\ge 4$, then by Theorem~\ref{power} and Lemma~\ref{tree-3-unif++},
\[\rho_{\mathcal{ABC}}(S_{5,k;1,1,1,1})=\rho_{\mathcal{ABC}}^\frac{4}{k}(S_{5,4;1,1,1,1})
<\rho_{\mathcal{ABC}}^\frac{4}{k}(S_{5,4;2,1,1})=\rho_{\mathcal{ABC}}(S_{5,k;2,1,1}).\]
In either case, we have $G\cong S_{m,k;m-3,1,1}$.

Suppose that $m\ge 6$. Suppose that $G\not\cong S_{m,k;m-3,1,1}$.  It suffices to show that $\rho_{\mathcal{ABC}}(G)<\rho_{\mathcal{ABC}}(S_{m,k;m-3,1,1})$.

As $G$ is a non-power $k$-uniform hypertree, the diameter of $G$ is at least $3$. As $G$ is a non-power $k$-uniform hypertree and is different from $S_{m,k;m-3,1,1}$, the maximum degree $G$ is at most $m-3$.

\noindent
{\bf  Case 1.}  The diameter of $G$ is either $3$ or $4$, and the maximum degree of $G$ is $m-3$.

Suppose first  that the diameter of $G$ is $3$.
Then $G\cong S_{m,k;m-4,2,1}$, or $S_{m,k;m-4,1,1,1}$ when $k\geq 4$.
By Theorem~\ref{power} and Lemma~\ref{tree-3-unif}(i),
\[
\rho_{\mathcal{ABC}}(S_{m,k;m-4,2,1})=\rho_{\mathcal{ABC}}^\frac{3}{k}(T_{m,1})
<\rho_{\mathcal{ABC}}^\frac{3}{k}(S_{m,3;m-3,1,1})=\rho_{\mathcal{ABC}}(S_{m,k;m-3,1,1}).
\]
If $k\geq 4$, then by Theorem~\ref{power} and Lemma~\ref{tree-3-unif++},
\[
\rho_{\mathcal{ABC}}(S_{m,k;m-4,1,1,1})=\rho_{\mathcal{ABC}}^\frac{4}{k}(S_{m,4;m-4,1,1,1})
<\rho_{\mathcal{ABC}}^\frac{4}{k}(S_{m,4;m-3,1,1})=\rho_{\mathcal{ABC}}(S_{m,k;m-3,1,1}).
\]
Thus $\rho_{\mathcal{ABC}}(G)<\rho_{\mathcal{ABC}}(S_{m,k;m-3,1,1})$.

Suppose next  that the diameter of $G$ is $4$.
Then $G\cong T_{m,2}^k$, $T_{m,3}^k$ or $T_{m,4}^k$.
By Theorem~\ref{power} and Lemma~\ref{tree-3-unif},
we have
\[
\max\{\rho_{\mathcal{ABC}}^\frac{3}{k}(T_{m,2}),\rho_{\mathcal{ABC}}^\frac{3}{k}(T_{m,3}),
\rho_{\mathcal{ABC}}^\frac{3}{k}(T_{m,4})\}<\rho_{\mathcal{ABC}}^\frac{3}{k}(S_{m,3;m-3,1,1})=\rho_{\mathcal{ABC}}(S_{m,k;m-3,1,1}).
\]
So  $\rho_{\mathcal{ABC}}(G)<\rho_{\mathcal{ABC}}(S_{m,k;m-3,1,1})$.

\noindent
{\bf Case 2.} The diameter of $G$ is at least $5$, or the maximum degree of $G$ at most $m-4$.

Note that if the diameter of $G$ is at least $5$, then the maximum degree of $G$ at most $m-4$. So
the maximum degree of $G$ at most $m-4$.
By Theorem \ref{med} and Lemma \ref{tree}(ii),
we have
\[\rho_{\mathcal{ABC}}(G)\leq \sqrt[k]{\frac{m-5}{m-4}}\rho_{\mathcal{A}}(G)\leq \sqrt[k]{\frac{m-5}{m-4}}\rho_{\mathcal{A}}((S_{m,k;m-4,2,1})).\]
So, we only need  to show that $\sqrt[k]{\frac{m-5}{m-4}}\rho_{\mathcal{A}}(S_{m,k;m-4,2,1})
<\rho_{\mathcal{ABC}}(S_{m,k;m-3,1,1})$.

Now, we calculate $\widetilde{\rho}=\rho_{\mathcal{A}}((S_{m,k;m-4,2,1}))$.
Let $v_1e_1v_2e_2v_3e_3v_4$ be a path of $S_{m,k;m-4,2,1}$ such that $v_2$ is of degree $m-3$ and
$v_3$ is of degree $2$.
Let $v_5$ be the vertex of degree $3$ in $e_2\setminus\{v_2,v_3\}$,
and $v_6$ and $v_7$ be pendant vertices in a pendant edge at $v_5$ and $e_2$, respectively.
Let ${\bf x}$ be the $k$-unit positive eigenvector of $\mathcal{A}(S_{m,k;m-4,2,1})$ corresponding to $\widetilde{\rho}$.
Let $x_i=x_{v_i}$ for $i=1,\ldots,7$.
By Lemma~\ref{ABC-permutation} and the  $(\widetilde{\rho}, {\bf x})$-eigenequations of $\mathcal{A}(S_{m,k;m-4,2,1})$, we have
\begin{align*}
\widetilde{\rho} x_1^{k-1}&=x_1^{k-2}x_2,\\
\widetilde{\rho} x_2^{k-1}&=(m-4)x_1^{k-1}+x_3x_5x_7^{k-3},\\
\widetilde{\rho} x_3^{k-1}&=x_2x_5x_7^{k-3}+x_4^{k-1},\\
\widetilde{\rho} x_4^{k-1}&=x_3x_4^{k-2},\\
\widetilde{\rho} x_5^{k-1}&=x_2x_3x_7^{k-3}+2x_6^{k-1},\\
\widetilde{\rho} x_6^{k-1}&=x_5x_6^{k-2}\\
\widetilde{\rho} x_7^{k-1}&=x_2x_3x_5x_7^{k-4}.
\end{align*}
By similar argument in the proof of Lemma~\ref{tree-3-unif} (i),
it is obtainable that $\widetilde{\rho}^{3k}-m\widetilde{\rho}^{2k}+(3m-10)\widetilde{\rho}^k-2m+8=0$, and so $\widetilde{\rho}$ is the largest root of $h(t^k)=0$,
where
\begin{equation}\label{za}
 h(t)=t^3-mt^2+(3m-10)t-2m+8.
\end{equation}

Note that
\begin{align*}
h(m-2)&=m^2-10m+20>0,\\
h(m-3)&=-(3m-10)<0,\\
h\left(\frac{m-\sqrt{m^2-9m+60}}{3}\right)&=\frac{2(m^2-9m+30)\sqrt{m^2-9m+60}}{9}+\frac{3m^2-29m+72}{3}\\
&> 0.
\end{align*}
So the largest root of $h(t)$ lies in $(m-3,m-2)$.
It follows that $\widetilde{\rho}\in (\sqrt[k]{m-3},\sqrt[k]{m-2})$.

Note that by \eqref{Eq0803-1},
\begin{equation}\label{Eq0911-1}
\eta\left(\left(\sqrt[k]{\frac{m-5}{m-4}}t\right)^k\right)=\frac{q\left(t^k\right)}{4(m-2)(m-4)^3},
\end{equation}
where
\begin{align*}
q(t)&=4(m-2)(m-5)^3t^3-(4m^2-19m+27)(m-5)^2(m-4)t^2\\
&\quad +(m-4)^2(4m^2-23m+34)(m-5)t-(m-3)^2(m-4)^3.
\end{align*}
By easy calculation and using \eqref{za}, we have
\[
q(t)=4(m-2)(m-5)^3 h(t)+r(t),
\]
where
\begin{align*}
r(t)&=t^2(7m^2 - 63m + 108)(m - 5)^2\\
&\quad -t(m - 5)(8m^4 - 129m^3 + 738m^2 - 1760m + 1456)\\
&\quad +(m - 4)(7m^4 - 122m^3 + 767m^2 - 2032m + 1856).
\end{align*}
Consider the axis of symmetry of the quadratic function $r(t)$ on $t$.
As
\[ \frac{8m^4 - 129m^3 + 738m^2 - 1760m + 1456}{2(7m^2 - 63m + 108)(m - 5)}<m-3,\]
$r(t)$ is strictly  increasing for $t\in (m-3,+\infty)$.
Let $s(m)=-r(m-2)$, i.e.,
\[
s(m)=m^6 - 31m^5 + 398m^4 - 2632m^3 + 9284m^2 - 16356m + 11184.
\]
Note that
\begin{align*}
s^{(1)}(m)&=6m^5 - 155m^4 + 1592m^3 - 7896m^2 + 18568m - 16356,\\
s^{(2)}(m)&=30m^4 - 620m^3 + 4776m^2 - 15792m + 18568,\\
s^{(3)}(m)&=120m^3 - 1860m^2 + 9552m - 15792,\\
s^{(4)}(m)&=360m^2 - 3720m + 9552,\\
s^{(5)}(m)&=720m - 3720,\\
s^{(6)}(m)&=720.
\end{align*}
It follows that $s^{(5-i)}(m)$ is strictly increasing for $m\geq 6$ as $s^{(6-i)}(6)>0$ with $i=1,\ldots, 6$.
Then  $s(m)\geq s(6)>0$ for $m\geq 6$, so
$r(m-2)=-s(m)<0$.
As $r(t)$ is strictly increasing for $t\in (m-3,m-2)$,
we have $r(t)<0$ for $t\in (m-3 ,m-2)$.

Let
$\widetilde{t}=\sqrt[k]{\frac{m-5}{m-4}}\widetilde{\rho}$.
Recall that $\widetilde{\rho}^k\in (m-3,m-2)$ and $h(\widetilde{\rho}^k)=0$.
Then by \eqref{Eq0911-1},
\begin{align*}
\eta\left(\widetilde{t}^k\right)
&=\frac{q\left( \widetilde{\rho}^k\right)}{4(m-2)(m-4)^3}\\
&=\frac{4(m-2)(m-5)^3 h(\widetilde{\rho}^k)+r(\widetilde{\rho}^k)}{4(m-2)(m-4)^3}\\
&=\frac{r(\widetilde{\rho}^k)}{4(m-2)(m-4)^3}\\
&<0,
\end{align*}
So $\widetilde{t}$ is less than the largest root of $\eta(t^k)=0$,
i.e., $ \sqrt[k]{\frac{m-5}{m-4}}\rho_{\mathcal{A}}((S_{m,k;m-4,2,1}))
<\rho_{\mathcal{ABC}}(S_{m,k;m-3,1,1})$.
Thus $\rho_{\mathcal{ABC}}(G)<\rho_{\mathcal{ABC}}(S_{m,k;m-3,1,1})$.
\end{proof}

\section{ABC spectral radius of unicyclic hypergraphs}

For integers $m\geq 2$, $k\geq 3$, $g=2,3$ and  $a_i$ for $1\leq i\leq k$ with  $0\leq a_i\leq m-g$ and $\sum_{i=1}^ka_i=m-g$,
let $U_{m,k,g}(a_1,\ldots, a_k)$ be the unicyclic graph obtained from a cycle
$u_1e_1u_2\dots u_ge_gu_1$
by adding $a_i$ pendant edges at $v_i$, where $e_1=\{v_1,\ldots, v_k\}$, $v_1=u_1$ and $v_k=u_2$.
Then $U_{m,g}^{(k)}\cong U_{m,k,g}(m-g, 0,\ldots, 0)$. It is evident that $U_{m,3}^{(k)}\cong U_{m,3}^k$.

\begin{Lemma}\label{unicyclic-2}
Let $k\geq3$, $g=2,3$, $m\geq 3$ and $a_i$ for $1\leq i\leq k$ be integers such that $a_1\geq a_k\geq 0$, $a_2\geq \dots\geq a_{k-1}\geq 0$ and $\sum_{i=1}^ka_i=m-g$.
Then
\[
\rho_{\mathcal{ABC}}(U_{m,k,g}(a_1,\ldots, a_k))\leq \rho_{\mathcal{ABC}}(U_{m,g}^{(k)})
\]
with equality if and only if $a_1=m-g$ and $a_2=\dots=a_k=0$.
\end{Lemma}

\begin{proof}
Denote by  $\mathcal{U}_{m,k,g}$ the class of hypergraphs $U_{m,k,g}(a_1,\ldots, a_k)$ with $a_1\geq a_k\geq 0$, $a_2\geq \dots \geq a_{k-1}\geq 0$ and $\sum_{i=1}^ka_i=m-g$.
Let $G=U_{m,k,g}(a_1,\ldots, a_k)$ be a hypergraph in $\mathcal{U}_{m,k,g}$ with maximum ABC spectral radius.
Let ${\bf x}$ be the $k$-unit positive eigenvector of $\mathcal{ABC}(G)$ corresponding to $\rho_{\mathcal{ABC}}(G)$. Then
by Lemma~\ref{R-qe},
$\rho_{\mathcal{ABC}}(G)=\mathcal{ABC}(G){\bf x}^k$.
Let $u_1e_1u_2\dots u_ge_g u_1$ be the cycle of $G$ as defined, where $e_1=\{v_1,\ldots, v_k\}$, $v_1=u_1$ and $v_k=u_2$.
Let $v'_i$ with $1\leq i\leq k$ be a pendant vertex in a pendant edge at $v_i$,
 $u$ be a pendant vertex in $e_2$, and $v$ be a pendant vertex in $e_3$ when $g=3$.
By Lemma~\ref{ABC-permutation}, the entry of ${\mathbf x}$ corresponding to each pendant vertex in
an edge is the same.

It is evident that $a_1\le m-g$.
Suppose that $a_1\leq m-g-1$. Then $a_s\ge 1$ for some $s=2, \dots, k$.
Let $H$ be the unicyclic hypergraph obtained from $G$ by moving all pendant edges from $v_j$ to $v_1$, where $2\leq j\leq k$.

Assume that  $x_{v_i}=\max_{1\leq j\leq k}x_{v_j}$, where $1\leq i\leq k$.

Suppose that $g=2$.
Let $\mathbf{y}$ be a vector such that $y_{v_1}=x_{v_i}$, $y_{v_i}=x_{v_1}$ and $y_w=x_w$ for $w\in V(G)\setminus \{v_1,v_i\}$ if $i>1$ and $\mathbf{y}={\bf x}$ otherwise.
By Lemma \ref{R-qe},
$\rho_{\mathcal{ABC}}(H)\ge \mathcal{ABC}(G)\mathbf{y}^{k}$. So
\begin{align*}
&\quad \frac{1}{k}(\rho_{\mathcal{ABC}}(H)-\rho_{\mathcal{ABC}}(G))\\
&\geq \frac{1}{k}\mathcal{ABC}(H)\mathbf{y}^k-\frac{1}{k}\mathcal{ABC}(G){\bf x}^k\\
&=\sqrt[k]{\dfrac{\sum_{w\in e_1}d_{H}(w)-k}{\prod_{w\in e_1} d_{H}(w)}}y_{v_1}\dots y_{v_k}+\sqrt[k]{\dfrac{\sum_{w\in e_2}d_{H}(w)-k}{\prod_{w\in e_2} d_{H}(w)}}y_{v_1} y_{v_k}y_{u}^{k-2}\\
&\quad +\sum_{j=1}^k a_j\sqrt[k]{\dfrac{d_{H}(v_1)-1}{d_{H}(v_1)}}y_{v_1} y_{v'_j}^{k-1}\\
&\quad -\left(\sqrt[k]{\dfrac{\sum_{w\in e_1}d_G(w)-k}{\prod_{w\in e_1} d_G(w)}}x_{v_1}\dots x_{v_k}+\sqrt[k]{\dfrac{\sum_{w\in e_2}d_G(w)-k}{\prod_{w\in e_2} d_G(w)}}x_{v_1} x_{v_k}x_{u}^{k-2}\right.\\
&\quad \left.+\sum_{j=1}^k a_j\sqrt[k]{\dfrac{d_{G}(v_j)-1}{d_{G}(v_j)}}x_{v_j} x_{v'_j}^{k-1}\right)\\
&=\sqrt[k]{\frac{1}{2}}x_{v_1}\dots x_{v_k}+\sqrt[k]{\frac{1}{2}}x_{v_i}x_{v_k}x_{u}^{k-2}
+\sum_{j=1}^k a_j\sqrt[k]{\frac{m-1}{m}}x_{v_i} x_{v'_j}^{k-1}\\
&\quad -\sqrt[k]{\dfrac{m}{(a_1+2)(a_k+2)\prod_{j=2}^{k-1}(a_j+1)}}x_{v_1}\dots x_{v_k}\\
&\quad -\sqrt[k]{\frac{a_1+a_k+2}{(a_1+2)(a_k+2)}}x_{v_1} x_{v_k}x_{u}^{k-2}-a_1\sqrt[k]{\frac{a_1+1}{a_1+2}}x_{v_1} x_{v'_1}^{k-1}\\
&\quad -\sum_{j=2}^{k-1} a_j\sqrt[k]{\frac{a_j}{a_j+1}}x_{v_j} x_{v'_j}^{k-1}-a_k\sqrt[k]{\frac{a_k+1}{a_k+2}}x_{v_k} x_{v'_k}^{k-1}\\
&=\left(\sqrt[k]{\frac{1}{2}}-\sqrt[k]{\dfrac{m}{(a_1+2)(a_k+2)\prod_{j=2}^{k-1}(a_j+1)}}\right)x_{v_1}\dots x_{v_k}\\
&\quad +\left(\sqrt[k]{\frac{1}{2}}x_{v_i}-\sqrt[k]{\frac{a_1+a_k+2}{(a_1+2)(a_k+2)}}x_{v_1}\right)x_{v_k}x_{u}^{k-2}\\
&\quad +a_1\left(\sqrt[k]{\frac{m-1}{m}}x_{v_i}-\sqrt[k]{\frac{a_1+1}{a_1+2}}x_{v_1}\right) x_{v'_1}^{k-1}\\
&\quad +\sum_{j=2}^{k-1}a_j\left(\sqrt[k]{\frac{m-1}{m}}x_{v_i}-\sqrt[k]{\frac{a_j}{a_j+1}}x_{v_j}\right) x_{v'_j}^{k-1}\\
&\quad +a_k\left(\sqrt[k]{\frac{m-1}{m}}x_{v_i}-\sqrt[k]{\frac{a_k+1}{a_k+2}}x_{v_k} \right) x_{v'_k}^{k-1}.
\end{align*}
By direct checking, we have
\[
\sqrt[k]{\frac{1}{2}}\geq \sqrt[k]{\frac{a_1+a_k+2}{(a_1+2)(a_k+2)}}.
\]
As $m=\sum_{j=1}^ka_j+2$, it is easy to see $(a_1+2)(a_k+2)\prod_{i=2}^{k-1}(a_i+1)\ge 2m$, so we have
\[
\sqrt[k]{\frac{1}{2}}\geq \sqrt[k]{\frac{m}{(a_1+2)(a_k+2)\prod_{i=2}^{k-1}(a_i+1)}}.
\]
As a function of $t$, $\frac{t}{t+1}$ is strictly increasing for $t>0$, so
\[
\sqrt[k]{\frac{m-1}{m}}> \max\left\{\sqrt[k]{\frac{a_1+1}{a_1+2}}, \sqrt[k]{\frac{a_2}{a_2+1}}, \ldots, \sqrt[k]{\frac{a_{k-1}}{a_{k-1}+1}}, \sqrt[k]{\frac{a_k+1}{a_k+2}}\right\}.
\]
By these inequalities and the above estimate for $\frac{1}{k}(\rho_{\mathcal{ABC}}(H)-\rho_{\mathcal{ABC}}(G))$, we have
\begin{align*}
\frac{1}{k}(\rho_{\mathcal{ABC}}(H)-\rho_{\mathcal{ABC}}(G)) & \ge \begin{cases}
 a_s\left(\sqrt[k]{\dfrac{m-1}{m}}x_{v_i}-\sqrt[k]{\dfrac{a_s}{a_s+1}}x_{v_s}\right) x_{v'_s}^{k-1} & \mbox{if } s<k\\ \\
a_s\left(\sqrt[k]{\dfrac{m-1}{m}}x_{v_i}-\sqrt[k]{\dfrac{a_k+1}{a_k+2}}x_{v_s} \right) x_{v'_s}^{k-1} &
\mbox{if } s=k
\end{cases}\\
&>0,
\end{align*}
so $\rho_{\mathcal{ABC}}(H)>\rho_{\mathcal{ABC}}(G)$, which is a contradiction. Thus  $a_1=m-2$ and $a_2=\dots=a_k=0$.

Suppose next that $g=3$.
First, suppose that $1\leq i\leq k-1$.
Let $\mathbf{y}$ be a vector such that $y_{v_1}=x_{v_i}$, $y_{v_i}=x_{v_1}$ and $y_w=x_w$ for $w\in V(G)\setminus \{v_1,v_i\}$ if $i>1$ and $\mathbf{y}={\bf x}$ otherwise.
Note that
$\sqrt[k]{\frac{1}{2}}\geq \sqrt[k]{\frac{a_1+a_k+2}{(a_1+2)(a_k+2)}}$,
$\sqrt[k]{\frac{1}{2}}\geq \sqrt[k]{\frac{m-1}{(a_1+2)(a_k+2)\prod_{i=2}^{k-1}(a_i+1)}}$
and $\sqrt[k]{\frac{m-2}{m-1}}> \max\left\{\sqrt[k]{\frac{a_1+1}{a_1+2}}, \sqrt[k]{\frac{a_2}{a_2+1}}, \ldots, \sqrt[k]{\frac{a_{k-1}}{a_{k-1}+1}},\right.$ $\left.\sqrt[k]{\frac{a_k+1}{a_k+2}}\right\}$.
By Lemma \ref{R-qe}, we have
\begin{align*}
&\quad \frac{1}{k}(\rho_{\mathcal{ABC}}(H)-\rho_{\mathcal{ABC}}(G))\\
&\geq \frac{1}{k}\mathcal{ABC}(H)\mathbf{y}^k-\frac{1}{k}\mathcal{ABC}(G){\bf x}^k\\
&=\sqrt[k]{\frac{1}{2}}y_{v_1}\dots y_{v_k}+\sqrt[k]{\frac{1}{2}}y_{u_3}y_{v_k}y_{u}^{k-2}
+\sqrt[k]{\frac{1}{2}}y_{u_3}y_{v_1}y_{v}^{k-2}\\
&\quad +\sum_{j=1}^k a_j\sqrt[k]{\frac{m-2}{m-1}}y_{v_1} y_{v'_j}^{k-1}-\sqrt[k]{\frac{m}{(a_1+2)(a_k+2)\prod_{j=2}^{k-1}(a_j+1)}}x_{v_1}\dots x_{v_k}\\
&\quad -\sqrt[k]{\frac{1}{2}}x_{u_3} x_{v_k}x_{u}^{k-2}-\sqrt[k]{\frac{1}{2}}x_{u_3} x_{v_1}x_{v}^{k-2}-a_1\sqrt[k]{\frac{a_1}{a_1+1}}x_{v_1}x_{v'_1}^{k-1}\\
&\quad -\sum_{j=2}^{k-1}a_j\sqrt[k]{\frac{a_j}{a_j+1}}x_{v_j}x_{v'_j}^{k-1}-a_k\sqrt[k]{\frac{a_k}{a_k+1}}x_{v_k}x_{v'_k}^{k-1}\\
&=\left(\sqrt[k]{\frac{1}{2}}-\sqrt[k]{\frac{m}{(a_1+2)(a_k+2)\prod_{j=2}^{k-1}(a_j+1)}}\right)x_{v_1}\dots x_{v_k}\\
&\quad +\sqrt[k]{\frac{1}{2}}\left(x_{v_i}-x_{v_1}\right)x_{u_3}x_{v}^{k-2}+a_1\left(\sqrt[k]{\frac{m-2}{m-1}}x_{v_i}
-\sqrt[k]{\frac{a_1+1}{a_1+2}}x_{v_1}\right) x_{v'_1}^{k-1}\\
&\quad +\sum_{j=2}^{k-1}a_j\left(\sqrt[k]{\frac{m-2}{m-1}}x_{v_i}-\sqrt[k]{\frac{a_j}{a_j+1}}x_{v_j}\right) x_{v'_j}^{k-1}\\
&\quad +a_k\left(\sqrt[k]{\frac{m-2}{m-1}}x_{v_i}-\sqrt[k]{\frac{a_k+1}{a_k+2}}x_{v_k} \right) x_{v'_k}^{k-1}\\
&>0.
\end{align*}
So $\rho_{\mathcal{ABC}}(H)>\rho_{\mathcal{ABC}}(G)$, a contradiction.
Now, suppose that $i=k$.
Let $\mathbf{z}$ be a vector such that $z_{v_1}=x_{v_k}$, $z_{v_k}=x_{v_1}$, $z_w=x_v$ for $w\in e_2\setminus\{u_2,u_3\}$,  $z_w=x_u$ for $w\in e_3\setminus\{u_1,u_3\}$ and $z_w=x_w$ for $w\in V(G)\setminus (e_1\cup e_2\cup e_3\setminus\{u_3\})$.
By Lemma \ref{R-qe}, we have
\begin{align*}
&\quad \frac{1}{k}(\rho_{\mathcal{ABC}}(H)-\rho_{\mathcal{ABC}}(G))\\
&\geq \frac{1}{k}\mathcal{ABC}(H)\mathbf{z}^k-\frac{1}{k}\mathcal{ABC}(G){\bf x}^k\\
&=\sqrt[k]{\frac{1}{2}}z_{v_1}\dots z_{v_k}+\sqrt[k]{\frac{1}{2}}z_{u_3}z_{v_k}z_{u}^{k-2}
+\sqrt[k]{\frac{1}{2}}z_{u_3}z_{v_1}z_{v}^{k-2}\\
&\quad +\sum_{j=1}^k a_j\sqrt[k]{\frac{m-2}{m-1}}z_{v_1} z_{v'_j}^{k-1}-\sqrt[k]{\frac{m}{(a_1+2)(a_k+2)\prod_{j=2}^{k-1}(a_j+1)}}x_{v_1}\dots x_{v_k}\\
&\quad -\sqrt[k]{\frac{1}{2}}x_{u_3} x_{v_k}x_{u}^{k-2}-\sqrt[k]{\frac{1}{2}}x_{u_3} x_{v_1}x_{v}^{k-2}-a_1\sqrt[k]{\frac{a_1}{a_1+1}}x_{v_1}x_{v'_1}^{k-1}\\
&\quad -\sum_{j=2}^{k-1}a_j\sqrt[k]{\frac{a_j}{a_j+1}}x_{v_j}x_{v'_j}^{k-1}-a_k\sqrt[k]{\frac{a_k}{a_k+1}}x_{v_k}x_{v'_k}^{k-1}\\
&=\left(\sqrt[k]{\frac{1}{2}}-\sqrt[k]{\frac{m}{(a_1+2)(a_k+2)\prod_{j=2}^{k-1}(a_j+1)}}\right)x_{v_1}\dots x_{v_k}\\
&\quad +a_1\left(\sqrt[k]{\frac{m-2}{m-1}}x_{v_k}-\sqrt[k]{\frac{a_1+1}{a_1+2}}x_{v_1}\right) x_{v'_1}^{k-1}\\
&\quad +\sum_{j=2}^{k-1}a_j\left(\sqrt[k]{\frac{m-2}{m-1}}x_{v_k}-\sqrt[k]{\frac{a_j}{a_j+1}}x_{v_j}\right) x_{v'_j}^{k-1}\\
&\quad +a_k\left(\sqrt[k]{\frac{m-2}{m-1}}x_{v_k}-\sqrt[k]{\frac{a_k+1}{a_k+2}}x_{v_k} \right) x_{v'_k}^{k-1}\\
&>0.
\end{align*}
So $\rho_{\mathcal{ABC}}(H)>\rho_{\mathcal{ABC}}(G)$, also a contradiction.
Thus $a_1=m-3$ and $a_2=\dots=a_k=0$.
\end{proof}

\begin{proof}[{\bf Proof of Theorem~\ref{abc7-18-1}}]
First, we calculate $\rho_{\mathcal{ABC}}(U_{m,2}^{(k)})$.
Let $v_1$ be the vertex of degree $m$ in $U_{m,2}^{(k)}$.
Let $v_1e_1v_2e_2v_1$ be the cycle of $U_{m,2}^{(k)}$.  Let $v_0$ and $v_3$ be pendant  vertices in a pendant edge and in $e_1$ of $U_{m,2}^{(k)}$, respectively.
Let ${\bf x}$ be the $k$-unit positive eigenvector  of $\mathcal{ABC}(U_{m,2}^{(k)})$ corresponding to $\rho=\rho_{\mathcal{ABC}}(U_{m,2}^{(k)})$.
Let $x_i=x_{v_i}$ for $i=0,1,2,3$.
By Lemma~\ref{ABC-permutation} and the $(\rho, \mathbf{y})$-eigenequations of $\mathcal{ABC}(U_{m,2}^{(k)})$, we have
\begin{align}
\rho x_0^{k-1}&=\sqrt[k]{\frac{m-1}{m}}y_0^{k-2}x_1,\label{21a1}\\
\rho x_1^{k-1}&=(m-2)\sqrt[k]{\frac{m-1}{m}}x_0^{k-1}+2\cdot \sqrt[k]{\frac{1}{2}}x_2 y_3^{k-2},\label{21a2}\\
\rho x_2^{k-1}&=2\cdot \sqrt[k]{\frac{1}{2}}x_1 x_3^{k-2},\label{21a3}\\
\rho x_3^{k-1}&=\sqrt[k]{\frac{1}{2}}y_1x_2x_3^{k-3}. \label{21a4}
\end{align}
By~\eqref{21a1}, we have $x_0=\frac{\sqrt[k]{\frac{m-1}{m}}x_1}{\rho}$.
By~\eqref{21a3} and~\eqref{21a4}, we have $x_3=\frac{x_2}{2^\frac{1}{k}}$, which, together  with ~\eqref{21a3}, implies that that $x_2=\frac{2^\frac{1}{k}x_1}{\rho}$, so  $x_3=\frac{x_1}{\rho}$.
Now  eliminating  $x_0,x_2$ and $x_3$ from \eqref{21a2} and noting that $x_1>0$,
we have $\rho-\frac{\frac{(m-1)(m-2)}{m}}{\rho^{k-1}}-\frac{2}{\rho^{k-1}}=0$,
i.e., $\rho^k-\frac{(m-1)(m-2)+2m}{m}=0$, i.e.,  $\rho^k=m-1+\frac{2}{m}$.
It thus follows that  $\rho_{\mathcal{ABC}}\left(U_{m,2}^{(k)}\right)=\sqrt[k]{m-1+\frac{2}{m}}$.

Now, we prove the result.
It is trivial if $m=2$.  Suppose that $m\ge 3$.
Let $G$ be a $k$-uniform unicyclic hypergraph of size $m$ different from $U_{m,2}^{(k)}$.

Suppose that the girth of $G$ is $2$.
Let $u_1e_1u_2e_2u_1$ be the cycle of $G$. Note that there is no edge different from $e_1$ and $e_2$
containing two vertices in $e_1\cup e_2$.
If there is an edge containing no vertex in $e_1\cup e_2$,
or there are two edges one containing a vertex in  $e_1\setminus \{u_1,u_2\}$ and the other containing a vertex in $e_2\setminus \{u_1,u_2\}$,
then  for any edge $e$ of  $G$, we have
$\sum_{w\in e}d_w-k\leq m-1$.
So by Theorem \ref{LUpper}, $\rho_{\mathcal{ABC}}(G)\leq \sqrt[k]{m-1}<\rho_{\mathcal{ABC}}(U_{m,2}^{(k)})$.
Suppose that any edge different from $e_1$ and $e_2$
is a pendant edge at some vertex in $e_i$ with $i=1,2$, say in $e_1$.
Let $e_1=\{v_1,\ldots,v_k\}$, where $v_1=u_1$ and $v_k=u_2$.
Let $a_i$ be the number of pendant edges at $v_i$ for $1\leq i\leq k$. Then  $\sum_{i=1}^{k}a_i=m-2$.
Assume that $a_1\geq a_k\geq 0$ and $a_2\geq \dots\geq a_{k-1}\geq 0$.
Then $G\cong U_{m,k,2}(a_1,\ldots,a_k)$.
So by Lemma~\ref{unicyclic-2}, $\rho_{\mathcal{ABC}}(G)<\rho_{\mathcal{ABC}}(U_{m,2}^{(k)})$.

Suppose the girth of $G$ is at least $3$.
Then for any edge $e=\{i_1,\ldots, i_k\}$ in $G$, we have
$\sum_{w\in e} d_w-k\leq m-1$.
By Theorem \ref{LUpper}, we have $\rho_{\mathcal{ABC}}(G)\leq \sqrt[k]{m-1}<\rho_{\mathcal{ABC}}(U_{m,2}^{(k)})$.
We complete the proof.
\end{proof}


\begin{proof}[{\bf Proof of Theorem~\ref{abc9-7-1}}]
By the definition of $f(t)$ in~\eqref{eq9-7},
we have
\begin{align*}
f(\sqrt{m-1})&=(m-1)\left(\sqrt{m-1}-\frac{1}{\sqrt{2}}\right)-\frac{m^2-4m+5}{\sqrt{m-1}}+\frac{m^2-5m+6}{\sqrt{2}(m-1)}\\
&=2\sqrt{m-1}-\frac{3}{\sqrt{2}}-\frac{2}{\sqrt{m-1}}+\frac{\sqrt{2}}{m-1}\\
&>0\\
f(\sqrt{m-2})&=(m-2)\left(\sqrt{m-2}-\frac{1}{\sqrt{2}}\right)-\frac{(m-2)(m^2-4m+5)}{m-1}+\frac{m^2-5m+6}{\sqrt{2}(m-1)}\\
&<m^2-5m+6-\frac{(m-2)(m^2-4m+5)}{m-1}\\
&<0,\\
f(0)&=\frac{m^2-5m+6}{\sqrt{2}(m-1)}\\
&>0.
\end{align*}
Then all roots of $f(t)$ lie in $(-\infty, 0)$, $(0,\sqrt{m-2})$ and $(\sqrt{m-2},\sqrt{m-1})$, respectively.
So $\sqrt{m-2}<a_m<\sqrt{m-1}$.
It follows from Theorem~\ref{power} that $\sqrt[k]{m-2}<\rho_{\mathcal{ABC}}(U_{m,3}^k)<\sqrt[k]{m-1}$.

Now, we prove the result.
It is trivial if $m=3$.  Suppose that $m\ge 4$.
Let $G$ be a linear $k$-uniform unicyclic hypergraph of size $m$ different from $U_{m,3}^{k}$.

As $G$ is linear, its girth is at least three.
If the girth of $G$ is at least $4$,
then for any edge $e$ in $G$, we have
$\sum_{w\in e} d_{G}(w)-k\leq m-2$, so,
by Theorem \ref{LUpper}, we have $\rho_{\mathcal{ABC}}(G)\leq \sqrt[k]{m-2}<\rho_{\mathcal{ABC}}(U_{m,3}^{k})$.
Suppose that the girth of $G$ is $3$.
Let $u_1e_1u_ke_2u_{2k-1}e_3u_1$ be a cycle of length three in $G$, where $e_i=\{u_{(i-1)(k-1)+j}:j=1,\ldots, k\}$ for $i=1,2,3$, and $u_1=u_{3k-2}$.
Suppose that  there is an edge  containing no vertex in $e_1\cup e_2\cup e_3$,
or there are two edges, one  containing a vertex in $e_i\setminus \{u_{(i-1)(k-1)+1},u_{(i-1)(k-1)+k}\}$ and the other containing a vertex in  $e_j\setminus \{u_{(j-1)(k-1)+1},u_{(j-1)(k-1)+k}\}$, where $1\leq i<j\leq 3$.
Then  for any edge $e$ in $G$, we have
$\sum_{w\in e} d_{G}(w)-k\leq m-2$.
So by Theorem \ref{LUpper}, $\rho_{\mathcal{ABC}}(G)\leq \sqrt[k]{m-2}<\rho_{\mathcal{ABC}}(U_{m,3}^{k})$.
Suppose  that  each edge different from $e_1$, $e_2$ and $e_3$
is a pendant edge at some vertex in exactly one of $e_1,e_2,e_3$, say $e_1$.
Let $a_i$ be the number of pendant edges at $u_i$ for $1\leq i\leq k$. Then $\sum_{i=1}^ka_i=m-3$.
Assume that $a_1\geq a_k\geq 0$ and $a_2\geq \dots\geq a_{k-1}\geq 0$.
Then $G\cong U_{m,k,3}(a_1,\ldots,a_k)$.
So by Lemma~\ref{unicyclic-2}, $\rho_{\mathcal{ABC}}(G)<\rho_{\mathcal{ABC}}(U_{m,3}^{k})$.
\end{proof}

\section{Concluding remarks}

We propose the ABC tensor of a uniform hypergraph as a generalization of the ABC matrix of a graph. We give tight lower and upper bounds for the ABC spectral radius of a uniform hypergraph and characterize the hypergraphs that attain these bounds.  We determine  the maximum ABC spectral radii of uniform hypertrees,  uniform non-hyperstar hypertrees and  uniform non-power  hypertrees of given size, as well as the maximum ABC spectral radii of unicyclic uniform hypergraphs and linear unicyclic uniform hypergraphs of given size, respectively. We also characterize those uniform hypergraphs for which the maxima for the ABC spectral radii are actually attained in all cases.
We two examples to show that the case for the ABC spectral radius of $k$-uniform hypergraphs for $k\ge 3$ is quite different from the ordinary case with $k=2$.

A hyperpath is a hypertree with at most two pendant edges.
Denote by $P_{m,k}$ the $k$-uniform hyperpath with $m$ edges.
On one hand, we note from \cite{ZLG} that $P_{m,k}$ for $m\ge 2$ and $k\ge 2$ is the unique connected $k$-uniform hypertrees  that minimizes the spectral radius and  \cite{Chen3} that $P_{m,2}$ is the unique connected graph of $m$ edges that minimizes the ABC spectral radius.

%
%

A pendant vertex of a hypergraph is a vertex of degree one. Let $G$ be a $k$-uniform hypergraph with $u\in V(G)$ and $u_i\not\in V(G)$ for $i=2,\dots,k$. The hypergraph with vertex set $V(G)\cup\{u, \dots, u_k\}$ and edge set $E(G)\cup \{u, u_2, \dots, u_k\}$ is said to be obtained from $G$ by adding a new pendant edge $\{u, u_2, \dots, u_k\}$ at $u$.

\noindent
{\bf Example 7.1.}
Let $H_1$ be the $3$-uniform hypertree with $6$ edges obtained from $S_{2,3}$ by adding a new pendant edge at each
pendant vertex.
Let $V_1$ be the set of vertices of degree one of $H_1$. Let $V_2$ be the set of vertices of degree two that lie in a pendant edge of $H_1$.
Let $\mathbf{x}$ be the $k$-unit positive eigenvector corresponding to $\rho(H_1)$.
By Lemma~\ref{ABC-permutation}, the entry of $\mathbf{x}$ corresponding to each vertex in $V_i$ for $i=1,2$ is equal, which we denote by $x_i$.
Denote by $x_3$ the entry of $\mathbf{x}$ corresponding to the only vertex of degree two outside $V_2$.
Then
\[
\rho_{\mathcal{ABC}}(H_1) x_1^2=\left(\frac{1}{2}\right)^\frac{1}{3}x_1x_2,
\]
\[
\rho_{\mathcal{ABC}}(H_1) x_2^2=\left(\frac{1}{2}\right)^\frac{1}{3}x_1^2+\left(\frac{3}{8}\right)^\frac{1}{3}x_2x_3
\]
and
\[
\rho_{\mathcal{ABC}}(H_1) x_3^2=\left(\frac{3}{8}\right)^\frac{1}{3}x_2^2\cdot 2.
\]
So $\rho(H_1)$ is the largest root of the equation $f(t)=0$, where
\[
f(t)=t^3-\left(\frac{3}{4}\right)^\frac{1}{2}t^{\frac{3}{2}}-\frac{1}{2}.
\]
By Theorem~\ref{LUpper},  $\rho_{\mathcal{ABC}}(H_1)\geq 1$.
As \[f'(t)=3t^2-\frac{3}{2}\cdot\left(\frac{3}{4}\right)^\frac{1}{2}\cdot t^{\frac{1}{2}}=3t^{\frac{1}{2}}
\left(t^{\frac{3}{2}}-\frac{1}{2}\cdot\left(\frac{3}{4}\right)^\frac{1}{2}\right)>0\]
for $t\geq 1$,
$f(\rho)$ is  strictly increasing  for $t\in [1,+\infty)$.
As
\[
f(1)=\frac{1}{2}-\left(\frac{3}{4}\right)^\frac{1}{2}\approx -0.366025<0,
\]
and
\[
f\left(\sqrt[3]{2\cos^{2}\frac{\pi}{8}}\right)=2\cos^2\frac{\pi}{8}
-\left(\frac{3}{2}\right)^\frac{1}{2}\cdot \cos\frac{\pi}{8}-\frac{1}{2}\approx0.07559>0,
\]
there is a root of $f(t)=0$ in $\left(1,\sqrt[3]{2\cos^{2}\frac{\pi}{8}}\right)$.
It thus follows that $\rho_{\mathcal{ABC}}(H_1)<\sqrt[3]{2\cos^{2}\frac{\pi}{8}}=
\rho_{\mathcal{ABC}}(P_{6,3})$.

\noindent
{\bf Example 7.2.}
Let $H_2$ be the $4$-uniform hypertree with $12$ edges obtained from $S_{3,4}$ by adding a new pendant edge at each
pendant vertex.
Let $V_i$ be the set of vertices of degree $i$ of $H$, where $i=1,2,3$.
Let $\mathbf{x}$ be the $k$-unit positive eigenvector corresponding to $\rho(H_2)$.
By Lemma~\ref{ABC-permutation}, the entry of $x$ corresponding to each vertex in $V_i$ for $i=1,2,3$ is equal, which we denote by $x_i$. Then
\[
\rho_{\mathcal{ABC}}(H_2) x_1^3=\left(\frac{1}{2}\right)^\frac{1}{4}x_1^2x_2,
\]
\[
\rho_{\mathcal{ABC}}(H_2) x_2^3=\left(\frac{1}{2}\right)^\frac{1}{4}x_1^3+\left(\frac{1}{3}\right)^\frac{1}{4}x_2^2x_3
\]
and
\[
\rho_{\mathcal{ABC}}(H_2) x_3^3=\left(\frac{1}{3}\right)^\frac{1}{4}x_2^3\cdot 3.
\]
So $\rho_{\mathcal{ABC}}(H_2)$ is the largest root of the equation $f(t)=0$, where
\[
f(t)=t^4-\left(\frac{5}{8}\right)^\frac{1}{3}t^{\frac{8}{3}}-\frac{1}{2}.
\]
By Theorem~\ref{LUpper},  $\rho_{\mathcal{ABC}}(H_2)\geq 1$.
As \[f'(t)=4t^3-\frac{8}{3}\cdot\left(\frac{5}{8}\right)^\frac{1}{3}\cdot t^{\frac{5}{3}}=4t^{\frac{5}{3}}
\left(t^{\frac{4}{3}}-\frac{2}{3}\cdot\left(\frac{5}{8}\right)^\frac{1}{3}\right)>0\]
for $t\geq 1$,
$f(\rho)$ is  strictly increasing  for $t\in [1,+\infty)$.
As
\[
f(1)=\frac{1}{2}-\left(\frac{5}{8}\right)^\frac{1}{3}\approx -0.35499<0
\]
and
\[
f\left(\sqrt[4]{2\cos^{2}\frac{\pi}{14}}\right)=2\cos^2\frac{\pi}{14}
-\left(\frac{5}{2}\right)^\frac{1}{3}\cdot \cos^\frac{4}{3}\frac{\pi}{14}-\frac{1}{2}\approx0.08894>0,
\]
there is a root of $f(t)=0$ in $\left(1,\sqrt[4]{2\cos^{2}\frac{\pi}{14}}\right)$.
It thus follows that $\rho_{\mathcal{ABC}}(H_2)<\sqrt[4]{2\cos^{2}\frac{\pi}{14}}=
\rho_{\mathcal{ABC}}(P_{12,3})$.

\vspace{5mm}

\noindent {\bf Acknowledgement.} 
This work was supported by National Natural Science Foundation of China (Nos.~12071158 and~11801410).


\begin{thebibliography}{99}




\bibitem {CPZ} K.C. Chang, K. Pearson,  T. Zhang,
Perron-Frobenius theorem for nonnegative tensors,
Commun. Math. Sci. 6 (2008) 507--520.


\bibitem{Chen2} X. Chen,
On extremality of ABC  spectral radius of a tree, Linear Algebra Appl.  564  (2019) 159--169.

\bibitem{Chen3} X. Chen,
A note on the ABC spectral radius of graphs, Linear Multilinear Algebra
70 (2022) 775--786.

\bibitem{Chung} F.R.K. Chung, Spectral Graph Theory, American Math. Soc., Providence, 1997.




\bibitem{CD} J. Cooper, A. Dutle,
Spectra of uniform hypergraphs, Linear Algebra Appl.  436  (2012) 3268--3292.

\bibitem{Da}
K. Das, M.A.  Mohammed, I.  Gutman, K.A. Atan,
Comparison between atom-bond connectivity indices of graphs,
MATCH Commun. Math. Comput. Chem. 76 (2016) 159--170.

%

\bibitem{Est} E. Estrada,
The ABC matrix, J. Math. Chem. 55 (2017) 1021--1033.

\bibitem{Est1} E. Estrada,
Atom-bond connectivity and the energetic of branched alkanes,
Chem. Phys. Lett. 463 (2008)
422--425.


\bibitem{Es22} E. Estrada,
Statistical-mechanical theory of topological indices,
Phys. A 602 (2022) 127612.

\bibitem{EB} E. Estrada, M. Benzi,
What is the meaning of the graph energy after all?
Discrete Appl. Math. 230 (2017) 71--77.


\bibitem{ETRG}
E. Estrada, L. Torres, L. Rodr\'{i}guez, I. Gutman,
An atom-bond connectivity index: modelling the enthalpy of formation of alkanes,
Indian J. Chem. 37A (1998) 849--855.





\bibitem{FGH} S. Friedland, S. Gaubert,  L. Han,
Perron-Frobenius theorem for nonnegative multilinear forms and extensions,
Linear Algebra Appl. 438 (2013) 738--749.





\bibitem{FGV}
B. Furtula, A. Graovac, D.  Vuki\v{c}evi\'{c},
Atom-bond connectivity index of trees, Discrete Appl. Math.  157  (2009) 2828--2835.


\bibitem{GF} I. Gutman, B.  Furtula,
Trees with smallest atom-bond connectivity index,
MATCH Commun. Math. Comput. Chem.  68  (2012) 131--136.

\bibitem{GFB} I. Gutman, B. Furtula, S. B. Bozkurt,
On Randi\'c energy, Linear Algebra Appl. 442 (2014) 50--57.

\bibitem{GK}
I. Gutman, E.V. Konstantinova, V.A. Skorobogatov,
Molecular hypergraphs and Clar structural formulas of benzenoid hydrocarbons,
ACH-Models Chem.
136 (1999) 539--548.






\bibitem{Gh} M. Ghorbani, X. Li, M. Hakimi-Nezhaad, J. Wang,
Bounds on the ABC  spectral radius and ABC  energy of graphs, Linear Algebra Appl.  598  (2020) 145--164.


\bibitem{HQ} S. Hu, L. Qi,
The Laplacian of a uniform hypergraph, J. Comb. Optim.  29 (2015) 331--366.

\bibitem{HQS}
S. Hu, L. Qi, J. Shao,
Cored hypergraphs, power hypergraphs and their Laplacian H-eigenvalues,
Linear Algebra Appl. 439 (2013) 2980--2998.


\bibitem{HDW}  H. Hua, K.  Das, H. Wang,
On atom-bond connectivity index of graphs, J. Math. Anal. Appl. 479 (2019) 1099--1114.


\bibitem{KS}
E.V. Konstantinova, V.A. Skorobogatov,
Molecular hypergraphs: The new representation of nonclassical molecular structures
with polycentric delocalized bonds,
J. Chem. Inf. Comput. Sci. 35 (1995) 472--478.


\bibitem{KS2}
E.V. Konstantinova, V.A. Skoroboratov,
Graph and hypergraph models of molecular structure: a comparative
analysis of indices,
J. Structure Chem. 39 (1998) 958--966.

%
%
%


%

%
%





\bibitem{LW} X. Li, J. Wang,
On the ABC spectra radius of unicyclic graphs, Linear Algebra Appl.  596  (2020) 71--81.


\bibitem{Lim} L. Lim,
Singular values and eigenvalues of tensors: a variational approach,
in: Proceedings of the First IEEE International Workshop on Computational Advances of Multi-Sensor Adaptive Processing, Puerto Vallarta, 2005, pp.~129--132.



\bibitem{N1} V. Nikiforov, Combinatorial methods for the spectral $p$-norm of hypermatrices, Linear Algebra Appl. 529 (2017) 324--354.

\bibitem{N2} V. Nikiforov, Analytic methods for uniform hypergraphs, Linear Algebra Appl. 457 (2014) 455--535.


\bibitem{PT} K. Pearson, T. Zhang,
On spectral hypergraph theory of the adjacency tensor, Graphs Combin.  30  (2014) 1233--1248.

\bibitem{Qi05} L. Qi, Eigenvalues of a real supersymmetric tensor,
J. Symbolic Comput. 40 (2005) 1302--1324.

\bibitem{Qi07}
L. Qi, Eigenvalues and invariants of tensors, J. Math. Anal. Appl. 325 (2007) 1363--1377.



\bibitem{Q3} L. Qi, H. Chen, Y. Chen,
Tensor Eigenvalues and Their Applications, Springer, Singapore, 2018.


\bibitem{Q2} L. Qi, Z.  Luo,
Tensor Analysis.
Spectral theory and special tensors, SIAM, Philadelphia, PA, 2017.


\bibitem{RSB} M. Rajesh Kannan, N. Shaked-Monderer, A. Berman,
On weakly irreducible nonnegative tensors and interval hull of some classes of tensors,
Linear Multilinear Algebra 64 (2016) 667--679.






%
%

\bibitem{YY} Y. Yang, Q. Yang,
Further results for Perron-Frobenius theorem for nonegative tensors,
SIAM J. Matrix Anal. Appl. 31 (2010) 2517--2530.


\bibitem{Yuan} X. Yuan, J. Shao, H. Shan, Ordering of some uniform supertrees with larger spectral
radii, Linear Algebra Appl. 495 (2016) 206--222.





\bibitem{YD} Y. Yuan, Z. Du,
The first two maximum ABC spectral radii of bicyclic graphs,
Linear Algebra Appl. 615 (2021) 28--41.


\bibitem{YZD} Y. Yuan, B. Zhou, Z. Du,
On large ABC spectral radii of unicyclic graphs, Discrete Appl. Math. 298 (2021) 56--65.

\bibitem{ZLG} J. Zhang, J. Li, H. Guo, Uniform hypergraphs with the first two smallest spectral radii, Linear Algebra Appl. 594 (2020) 71--80.





\bibitem{ZX} B. Zhou, R. Xing,
On atom-bond connectivity index, Z. Naturforsch. A 66 (2011) 61--66.

\bibitem{ZSWB} J. Zhou, L. Sun, W. Wang, C. Bu, Some spectral properties of uniform hypergraphs, Electron. J. Combin. 21 (2014) Paper 4.24. 





\end{thebibliography}
\end{document}